\documentclass[a4paper,smallmargins,11pt]{article}

\usepackage{philip_arxiv_layout}
\usepackage{bm}
\usepackage{bbm}
\usepackage{amssymb}
\usepackage{mathtools}
\usepackage{microtype}
\usepackage{nicefrac}

\newcommand{\wto}{\overset{w}{\to}}

\newcommand{\diff}{\mathrm{d}}

\DeclareMathOperator*{\BL}{\mathrm{BL}}

\DeclareMathOperator*{\Indep}{\perp\mkern-11mu\perp}
\DeclareMathOperator*{\nIndep}{\not\mkern-2mu\perp\mkern-11mu\perp}

\DeclareMathOperator*{\given}{|}

\newcommand{\I}{\mathbbm{1}}

\newcommand{\DD}{\mathbb{D}}
\newcommand{\EE}{\mathbb{E}}

\newcommand{\NN}{\mathbb{N}}
\newcommand{\PP}{\mathbb{P}}
\newcommand{\QQ}{\mathbb{Q}}

\newcommand{\RR}{\mathbb{R}}

\newcommand{\Bcal}{\mathcal{B}}

\newcommand{\Ncal}{\mathcal{N}}
\newcommand{\Ocal}{\mathcal{O}}
\newcommand{\Pcal}{\mathcal{P}}

\newcommand{\Xcal}{\mathcal{X}}
\newcommand{\Ycal}{\mathcal{Y}}
\newcommand{\Zcal}{\mathcal{Z}}

\usepackage{tikz}
\usetikzlibrary{positioning, arrows}
\usetikzlibrary{automata}
\usetikzlibrary{arrows}
\usetikzlibrary{arrows.meta, calc}
\usetikzlibrary{shapes,positioning,decorations.markings}
\usetikzlibrary{decorations.pathreplacing}
\usetikzlibrary{arrows,fit}
\usetikzlibrary{matrix,decorations.pathmorphing}
\usetikzlibrary{patterns}
\tikzstyle{var}=[circle,draw,thick,minimum size=20pt,inner sep=0pt]
\tikzstyle{vare}=[ellipse,draw,thick,minimum size=20pt,inner sep=0pt]
\tikzstyle{vari}=[shape=rectangle,draw,thick,minimum size=20pt,inner sep=0pt]
\tikzstyle{varh}=[circle,draw,thick,minimum size=20pt,inner sep=0pt,dashed]
\tikzstyle{arr}=[->,>=stealth',draw,thick]
\tikzstyle{arrh}=[->,>=stealth',draw,fill,thick,dashed]
\tikzstyle{biarr}=[<->,>=stealth',draw,thick]

\tikzstyle{ndint} = [draw, line width=1pt, color=teal, shape=rectangle, minimum size=20pt,inner sep=0pt, text=black]
\tikzstyle{ndout} = [draw, line width=1pt, shape=circle, color=blue, minimum size=20pt,inner sep=0pt, text=black]
\tikzstyle{ndlat} = [draw, line width=1pt, shape=circle, color=red, minimum size=20pt,inner sep=0pt, text=black]
\tikzstyle{arint} = [style={->,>=Latex,thick,teal}]
\tikzstyle{arout} = [style={->,>=Latex,thick,blue}]
\tikzstyle{arlout} = [style={->,>=Latex,thick,red}]
\tikzstyle{arlat} = [style={<->,>=Latex,thick,red}]
\tikzstyle{car} = [style={o->,>=Latex,thick,purple}]
\tikzstyle{carc} = [style={o-o,>=Latex,thick,purple}]
\tikzstyle{sarr}=[style={{Rays[n=6]}->,>=Latex,thick}]
\tikzstyle{sars}=[style={{Rays[n=6]}-{Rays[n=6]},thick}]

\usepackage{thmtools} 
\usepackage{thm-restate}

\author{%
	Philip Boeken\thanks{Department of Mathematics, VU Amsterdam, \url{p.a.boeken@vu.nl}}%
	\and%
	Eduardo Skapinakis\thanks{Carl Friedrich von Weizsäcker Zentrum, Universität Tübingen, Center for Mathematics and Applications (NOVA Math), NOVA School of Science and Technology (NOVA FCT), \url{eduardo.skapinakis@uni-tuebingen.de}}%
	\and%
	Konstantin Genin\thanks{Deptartment of Philosophy, University of Utah, \url{konstantin.genin@utah.edu}}%
	\and%
	Joris M.\ Mooij\thanks{Korteweg-de Vries Institute for Mathematics, University of Amsterdam, \url{j.m.mooij@uva.nl}}%
}

\date{\today}

\title{Topological Criteria for Hypothesis Testing with Finite-Precision Measurements}

\begin{document}

\maketitle

\begin{abstract}
	We establish topological necessary and sufficient conditions under which a pair of statistical hypotheses can be consistently distinguished when i.i.d.\ observations are recorded only to finite precision. To accommodate finite-precision data, we introduce \emph{finite-precision tests}: tests whose decision regions are open in the sample-space topology. We first show that, both for classical and finite-precision tests, the existence of such tests with finite-sample error control, asymptotic error control, or uniform convergence of the errors are all equivalent. A pair of null- and alternative hypotheses $H_0$ and $H_1$ admits a consistent finite-precision test if and only if both are $F_\sigma$ in the weak topology on the space of probability measures $W := H_0\cup H_1$. The hypotheses admit uniform error control under $H_i$ if and only if $H_i$ is closed in $W$, and admit uniformly consistent testing with \emph{bounded} precision under metric separation of $H_0$ and $H_1$.
These criteria imply that, without regularity assumptions, conditional independence is not consistently testable from finite-precision data when the conditioning space has no isolated points --- strengthening existing impossibility results to Polish sample spaces and showing that even pointwise consistency cannot be obtained. We introduce an equicontinuity assumption on the family of conditional distributions under which we recover consistent finite-precision testability of conditional independence with uniform error control under the null, provided sample spaces are Polish and the conditioning space is locally compact. The equicontinuity assumption is itself a finite-precision-testable hypothesis, so the resulting test for conditional independence is, in a precise sense, assumption-free.
\end{abstract}

\section{Introduction}
Whether a statistical hypothesis is testable remains an important question across the natural and social sciences. For example, judgements of conditional independence underlie many scientific inferences and are particularly fundamental in structural causal discovery \citep{spirtes1993causation}. Recent results demonstrate that, unless regularity assumptions are made, any test for conditional independence with Type-I error control at level $\alpha$ has power at most $\alpha$ against every alternative \citep{shah2020hardness}. Nevertheless, it was believed that conditional independence is consistently testable \citep{gyorfi2012strongly}. The true situation turns out to be more complicated \citep{neykov2021minimax}. These recent results in conditional independence testing are arrived at via ingenious ad hoc arguments. Despite recent developments, there remains no simple, unified criterion for characterizing all and only the testable statistical hypotheses. In this respect, statistics is in sharp contrast to the theory of computation, where the testable hypotheses (co-semidecidable sets) receive elegant complexity-theoretic characterizations \citep{kelly1996logic}. In this paper, we attempt to remedy this situation by laying out general topological conditions for statistical testability, generalizing results from \cite{dembo1994topological}, \cite{ermakov2017consistent}, \cite{genin2017topology} and \cite{kleijn2022frequentist}.

We demonstrate our results on the following elementary parametric examples since, contrary to many nonparametric hypotheses like conditional independence, their topological properties are evident.
Let $X_1, ..., X_n \sim \mathrm{Bernoulli}(p)$ i.i.d.\ and consider for given $\varepsilon>0$ the pairs of hypotheses
\begin{align}
	 & H_0: p \in [0,1]\cap \QQ        &  & H_1: p \in [0,1]\setminus \QQ.\label{eqn:ex:untestable}     \\
	 & H_0: p \in [0,1]\cap \QQ        &  & H_1: p \in [0,1]\cap (\QQ + \sqrt{2})\label{eqn:ex:fsigma}  \\
	 & H_0: p \in [0, 1/2]             &  & H_1: p \in (1/2, 1]       \label{eqn:ex:closed}             \\
	 & H_0: p \in [0, 1/2)             &  & H_1: p \in (1/2, 1]      \label{eqn:ex:clopen}              \\
	 & H_0: p \in [0, 1/2-\varepsilon) &  & H_1: p \in (1/2+\varepsilon, 1]    \label{eqn:ex:separated}
\end{align}
As might be intuitively clear, for these pairs of hypotheses one can achieve different consistency properties of the test, ranging from uniform consistency for (\ref{eqn:ex:separated}) to mere consistency for (\ref{eqn:ex:fsigma}), and (\ref{eqn:ex:untestable}) is not consistently testable. The various notions of consistency and error control are made precise in Section \ref{sec:testing_framework}, where we also show that existence of tests with many seemingly different formulations are in fact equivalent (Theorem \ref{thm:testable_equiv_error}). Our main results, presented in Section \ref{sec:characterisations}, state that these modes of testability are characterised by the topological properties of the hypotheses: whether they are $F_\sigma$ (i.e.\ a countable union of closed sets), closed or open, or clopen with respect to the subspace topology on $W := H_0\cup H_1$, or metrically separated.
We don't restrict to parametric hypotheses, but we consider arbitrary nonparametric hypotheses $H_0$ and $H_1$ as subsets of the space of Borel probability measures $\Pcal(\Xcal)$ on a separable metric space $\Xcal$. The topological characterisations are with respect to the weak topology on $\Pcal(\Xcal)$. The data is assumed to be i.i.d.\ from a single $\PP$ in either $H_0$ or $H_1$.
Notably, most of our results do not require any regularity conditions on the probability measures under consideration. However, it turns out that the weak topology characterises testability with \emph{some} type of regularity: that the critical regions of the tests are open --- a property that is useful when considering finite-precision measurements, to be further motivated in Section \ref{sec:finite_precision} below.


Our main application of these topological characterisations is to analyse the feasibility of conditional independence testing. For given sample spaces $\Xcal, \Ycal, \Zcal$ and set of probability measures $W\subseteq \Pcal(\Xcal \times \Ycal \times \Zcal)$ we consider
\begin{align}
	 & H_0:= \{\PP\in W : X\Indep_\PP Y\given Z\} &  & H_1:= \{\PP\in W : X\nIndep_\PP Y\given Z\}, \label{eqn:ex:ci}
\end{align}
where conditional independence, denoted with $X\Indep_\PP Y\given Z$, means that for all measurable $A$ and $B$ the factorisation $\PP(X\in A, Y\in B \given Z) = \PP(X\in A \given Z)\PP(Y\in B \given Z)$ holds $\PP(Z)$-almost surely. Conditional dependence is denoted with $X\nIndep_\PP Y\given Z$.
In Section \ref{sec:ci_untestable}, we show that conditional independence and conditional dependence are both dense in the weak topology on $\Pcal(\Xcal \times \Ycal \times \Zcal)$. It follows from the Baire category theorem that there does not exist a consistent FP-test for conditional independence when $W =\Pcal(\Xcal \times \Ycal \times \Zcal)$. Hence, conditional independence testing requires assumptions, that is, a restriction of the set $W$.
We introduce the space of distributions whose conditional distributions $\PP(X\given Z)$ or $\PP(Y\given Z)$ are equicontinuous, and investigate some properties in Section \ref{sec:ci_closed}. We show that this space is closed in the weak topology, and that conditional independence is closed in this space. In Section \ref{sec:ci_testing} we discuss the consequences for various modes of testability of conditional independence.

\subsection{Accommodating finite-precision measurements}\label{sec:finite_precision}
Given a test $\varphi_n :\Xcal^n\to\{0,1\}$ and a sample $x \in \Xcal^n$, the verdict of the test is given by the evaluation $\varphi_n(x)$. From now on, we write $\{\varphi_n=i\}$ for the region $\{x \in \Xcal^n : \varphi_n(x) = i\}.$ In standard presentations of hypothesis testing, the acceptance and rejection regions of a test are taken to be arbitrary \emph{measurable} sets. But is it really natural to consider a test that rejects if the sample point is rational-valued? Or if it is precisely $\pi$? If measurements of real-valued quantities can be made with only finite precision, such tests cannot be implemented.

We adopt the following idealisation. A measurement device, given a true sample value $x\in\Xcal^n$, reports a truncated value $x'\in\Xcal^n$ and a precision parameter $\varepsilon>0$ with $d_{\infty}(x,x')<\varepsilon$, where $d_{\infty}$ denotes the $\ell^\infty$ product metric on $\Xcal^n$. The empiricist observes the pair $(x',\varepsilon)$ but \emph{not} the true value $x$; what she knows after a single measurement is that $x\in B(x',\varepsilon)$, the open ball of radius $\varepsilon$ around the report. By requesting reports at decreasing precision $\varepsilon\downarrow 0$, the empiricist can in principle pin down $x$ to arbitrary accuracy.

Given a measurement $(x',\varepsilon)$ and a test $\varphi_n$, the empiricist declares verdict $i\in\{0,1\}$ if $B(x',\varepsilon)\subseteq\{\varphi_n=i\}$, in which case $\varphi_n(x)=i$ regardless of where $x$ sits inside $B(x',\varepsilon)$. If no such $i$ exists, the empiricist refines the precision and tries again. For this protocol to terminate at samples lying in $\{\varphi_n=0\}$ or $\{\varphi_n=1\}$ we require these critical regions to be open in the topology on $\Xcal^n$: openness of $\{\varphi_n=i\}$ guarantees that for any $x$ in this critical region, sufficiently small $\varepsilon$ yields $B(x',\varepsilon)\subseteq\{\varphi_n=i\}$. If $\Xcal^n$ is connected, $\{\varphi_n=0\}$ and $\{\varphi_n=1\}$ cannot both be nonempty, open, and cover the entire sample space; we therefore introduce the region $\{\varphi_n=2\}$ that recommends neither acceptance nor rejection of $H_0$, but suspension of judgement.\footnote{Throughout the paper, we assume that the data comes from a distribution in $H_0$ or $H_1$. Some of our mathematical results are more generally applicable to three-arm testing problems where the true $\PP$ comes from $H_0, H_1$ or $H_2$. However, in such cases one cannot verify based on finite-precision measurements whether a sample lands in $\{\varphi_n=2\}$, and one should take $H_2$ into account in the definitions of consistency and uniform error control (see Section \ref{sec:testing_framework}). In that case, the output $\{\varphi_n = 2\}$ is not restricted to the interpretation `suspension of judgement'.}
\begin{definition}
	A \emph{finite-precision test (FP-test)} is a measurable map $\varphi_n:\Xcal^n \to \{0,1,2\}$ such that $\{\varphi_n=0\}$ and $\{\varphi_n=1\}$ are open.
\end{definition}
Abusing terminology, we sometimes refer to a \emph{(FP-)testing sequence} $\varphi = (\varphi_n)_{n\in \NN}$ as a \emph{(FP-)test}.
For an FP-test, the output protocol terminates with the correct verdict whenever $x$ lies in the interior of any of the three regions; only at the boundary of $\{\varphi_n=2\}$ can refinement fail to halt, since at every such boundary point every measurement ball $B(x',\varepsilon)$ intersects one of the open decision regions. Then, why is it reasonable to restrict our attention to finite-precision tests?
The first thing to notice is that this requirement is already more realistic than the usual one in which we impose no conditions (besides measurability) on the complexity of the regions. Furthermore, positive results are stronger: if it is informative to learn that a hypothesis is testable with arbitrary (measurable) regions, it is even more informative to learn that it is testable with arbitrary but finite precision. Negative results are still informative: if you cannot test a hypothesis with arbitrary but finite precision, then you certainly cannot test it with bounded precision.

For a generic FP-test, the termination of the precision-refinement protocol may depend on $x$ and on $n$, with no a priori bound across the sample space. Real measurement devices, however, have a fixed bounded resolution.
The stronger condition, in which a single fixed precision works at every sample size, is captured by the following definition.
\begin{definition}\label{def:bp_test}
	A \emph{bounded-precision test (BP-test)} $(\varphi_n)_{n\in\NN}$ is a finite-precision test such that for some $\delta > 0$ we have $d_{\infty}(\{\varphi_n=0\}, \{\varphi_n=1\}) \geq \delta$ for all $n\in\NN$.\footnote{Here, $d_{\infty}$ denotes metric on $\Xcal^n$, taken to be the $\ell_\infty$ product of the metric on $\Xcal$.}
\end{definition}
For a BP-test, the empiricist may fix a precision $\varepsilon<\delta/2$, since then $B(x',\varepsilon)$ has diameter at most $2\varepsilon<\delta$ and therefore cannot simultaneously intersect $\{\varphi_n=0\}$ and $\{\varphi_n=1\}$. The only residual ambiguity is between a single decision region $\{\varphi_n=i\}$ ($i\in\{0,1\}$) and the suspension region $\{\varphi_n=2\}$. Under uniform convergence $\sup_{\PP\in H_0\cup H_1}\PP^n(\varphi_n=2)\to 0$ (c.f.\ Theorem \ref{thm:testable_iff_weakly_separated}), the empiricist's asymptotic guarantees are unaffected by how this residual ambiguity is resolved.

Every FP-test (and hence also every BP-test) can be transformed into a binary test by merging $\{\varphi_n=2\}$ with $\{\varphi_n=0\}$ or $\{\varphi_n=1\}$. Pointwise and uniform consistency are maintained by transforming it into a binary test in either way, while for maintaining one-sided error control one should be more careful. More details are given in Section \ref{sec:characterisations}.

The finite-precision measurement framework above is intrinsically topological, and instead of metric balls, one can consider other topological bases whose elements ought to correspond to finite-precision measurements. For more on topological and computational models of measurement, see \citet{moore1966interval,vickers1990topology,kelly1996logic,beggs2010computational,genin2017topology,resende2021abstract}.

\subsection{Related literature}
In a sufficiently regular parametric setting, \cite{berger1949distinct} provide necessary and sufficient conditions for the existence of a test with Type-I and Type-II error control with $\alpha=\beta = 1/2$.
A rather technical characterisation for the existence of uniformly consistent tests has been given by \cite{berger1951uniformly}. A topological characterisation has been given by \cite{lecam1960necessary} in terms of a topology of setwise convergence on the space $\cup_{n=1}^\infty \Pcal(\Xcal)^n$ of all $n$-fold products of probability measures on $\Xcal$. The downside of this condition is that it is hard to verify since this topology is not metrisable, not first countable and hence ``not very easily accessible'' \citep{lecam1960necessary} --- see also \cite{kleijn2022frequentist}.
In contrast, later work (and this paper as well) considers topological properties of $H_0, H_1$ as subsets of $\Pcal(\Xcal)$, instead of the space of all $n$-fold products.
\cite{cover1973determining} considered testing whether the bias of a coin is rational or irrational (example (\ref{eqn:ex:untestable})), and found that there exists a consistent test for testing the null of rational bias versus the alternative of irrational bias, if a subset of Lebesgue measure zero is removed from the alternative.
\cite{dembo1994topological} show that if $H_0$ and $H_1$ are disjoint sets which are $F_\sigma$ in the weak topology on $W = H_0 \cup H_1$, then there exists a strongly consistent test. For the converse direction, they note that the hypotheses $H_0 := \{\delta_x : x\in [0,1]\cap \QQ\}$ and $H_1 := \{\delta_x : x \in [0,1]\setminus \QQ\}$ are discernible by the test $\varphi_n(x_1, ..., x_n) := \I_{[0,1]\setminus \QQ} (x_1)$, but $H_1$ is not $F_\sigma$\footnote{This holds since $[0,1]\setminus \QQ$ is not $F_\sigma$ and $x\mapsto \delta_x$ is a homeomorphism (\citealp{bogachev2007measurevol2}, Lemma 8.9.2).} so \emph{some} regularity condition has to be imposed for such a topological characterisation of discernibility.
\cite{dembo1994topological} prove this characterisation under the assumption that every measure in $W$ has a $p>1$-integrable density, and \cite{kleijn2022frequentist} (Corollary 9.4.23) weakens this assumption to uniformly integrable densities. We show without any assumptions on the hypotheses $H_0$ and $H_1$ that consistent FP-testability is characterised by hypotheses that are $F_\sigma$ in the weak topology, so requiring the regions $\{\varphi_n = 0\}$ and $\{\varphi_n = 1\}$ to be open is \emph{the} regularity condition which resolves the issue raised by \cite{dembo1994topological}.
\cite{ermakov2017consistent} provides characterisation theorems similar to our Theorems \ref{thm:discernible_iff_weakly_fsigma}, \ref{thm:testable_iff_weakly_closed} and \ref{thm:testable_iff_weakly_separated}, but for binary tests without regularity conditions on the critical regions, in terms of the topology of setwise convergence.
For a comprehensive overview of the literature on consistent hypothesis testing, we refer to \cite{kleijn2022frequentist}, Chapter 9.

\cite{larsson2026complete} show that a test --- as general $[0,1]$-valued random variable on the underlying measurable space, so without i.i.d.\ assumption --- that satisfies the condition $\sup_{\PP \in H_0} \EE_\PP[\varphi] < \inf_{\PP \in H_1} \EE_\PP[\varphi]$ exists if and only if the weak$^*$ closures (in the space of bounded finitely additive measures) of the convex hulls of $H_0$ and $H_1$ are disjoint.

Our results in Section \ref{sec:characterisations} are largely based on the work by \cite{genin2017topology}, who work under the assumption that the sample space $\Xcal$ has a subbasis $\Ocal $ such that $\PP(\partial A) = 0$ for all $\PP\in W$ and $A\in \Ocal$ --- this assumption is satisfied if e.g., all measures in $W$ have a density with respect to some dominating measure (\citealp{bogachev2007measurevol2}, Proposition 8.2.8).
The results of \cite{genin2017topology} are applied to the study of causal discovery in \citet{genin2020statistical, genin2021statistical, genin2024success}.

Regarding conditional independence testing, \cite{shah2020hardness} and \cite{neykov2021minimax} show, in the setting where the variables are real-valued and the probability measures have densities, that no conditional independence test with Type-I error control and consistency under the alternative exists. \cite{lundborg2022conditional} generalise this to a specific setting where the samples are $L^2([0,1],\RR)$ functions, e.g.\ continuous-time stochastic processes. \cite{gyorfi2012strongly} provide a conditional independence test and prove that it is strongly consistent, but \cite{neykov2021minimax} point out a mistake in their proof, so it remains an open question whether conditional independence is consistently testable. We answer this question for FP-testability by showing that if $X,Y,Z$ take values in arbitrary Polish spaces $\Xcal, \Ycal, \Zcal$ where $\Zcal$ has no isolated points, there exists no consistent FP-test for conditional independence $X\Indep Y\given Z$.

In Section \ref{sec:ci_closed} we show that conditional independence is weakly closed under similar conditions as considered by \cite{barbie2014topology}; we provide a more direct, alternative proof. See Section \ref{sec:ci_closed_related_literature} for further discussion of related literature.
In Section \ref{sec:ci_testing}, we provide sufficient conditions for the consistent testability of conditional independence with and without uniform error control under $H_0$. Our conditions for testability with error control are similar to those considered by \cite{warren2021wasserstein} and \cite{neykov2021minimax}. To the best of our knowledge, our conditions for consistent testability are novel: we require that only one of the maps
\begin{align*}
	z & \mapsto \PP(X \given Z = z), \\
	z & \mapsto \PP(Y \given Z = z)
\end{align*}
satisfies a regularity condition. For a more in-depth comparison with existing literature, see Section \ref{sec:ci_related_literature}.

\subsection{The weak topology}\label{sec:weak_topology}
We will characterise FP-testability in terms of topological properties of the hypotheses $H_0, H_1\subseteq \Pcal(\Xcal)$, where $\Xcal$ is a separable metric space, and $\Pcal(\Xcal)$ denotes the set of probability measures on the Borel $\sigma$-algebra $\Bcal(\Xcal)$ on $\Xcal$. Recall that a sequence of probability measures $\PP_1(X), \PP_2(X), ...$ converges weakly to another probability measure $\PP(X)$, denoted with $\PP_n\wto \PP$, if $\int f(x)\diff\PP_n(x) \to \int f(x)\diff\PP(x)$ for all bounded continuous functions $f:\Xcal\to \RR$. The \emph{weak topology} on the space of Borel probability measures $\Pcal(\Xcal)$ is the smallest topology that makes the maps $\PP\mapsto \int f(x)\diff \PP$ continuous for all bounded continuous $f$.
Because $\Xcal$ is separable, the weak topology is separable and metrisable, for example by the \emph{bounded Lipschitz metric}\footnote{The bounded Lipschitz metric is also known as the \emph{Kantorovic-Rubinstein metric}. It is weaker than the total variation metric $d_{TV}$ and the \emph{$p$-Wasserstein metric} $W_p$: we have $d_{BL} \leq W_1 \leq W_p$, and strong equivalence of $d_{BL}$ and $W_1$ if the measures under consideration have bounded support (\citealp{bogachev2007measurevol2}, Theorem 8.10.45).} $d_{BL}$, defined by
\begin{equation*}
	d_{BL}(\PP_0(X), \PP_1(X)) := \sup\left\{\left|\int f \diff(\PP_0 - \PP_1) \right| : f\in \BL(\Xcal; \RR) \right\},
\end{equation*}
where $\BL(\Xcal; \RR) := \left\{f : \Xcal \to \RR ~\left|~ \sup_{x\neq x'}\frac{|f(x) - f(x')|}{d(x, x')} \leq 1 \text{ and } \|f\|_\infty \leq 1\right.\right\}$ (\citealp{bogachev2007measurevol2}, Theorem 8.3.2). If $\Xcal$ is complete, then $d_{BL}$ is complete as well (\citealp{bogachev2007measurevol2}, Theorem 8.10.43).
Since the weak topology is sequential, convergence in the weak topology coincides with weak convergence.
Weak convergence $\PP_n \wto \PP$ is equivalent to the condition that $\liminf_n\PP_n(A)\geq \PP(A)$ for all $A\subseteq \Xcal$ open, and to the condition that $\limsup_n\PP_n(B)\leq \PP(B)$ for all $B\subseteq \Xcal$ closed. This is also known as the `Portmanteau theorem' (\citealp{bogachev2007measurevol2}, Corollary 8.2.4(a)).
A subbasis for the weak topology is given by sets $\{\PP : \PP(A) > q\}$, with $A\subseteq \Xcal$ open and $q\in [0,1]$; see for example \cite{bogachev2007measurevol2}, Section 8.2.
This means that for any weakly open set $H \subseteq \Pcal(\Xcal)$, there exist open sets $A_{ij} \subseteq \Xcal$ and $q_{ij}\in [0,1]$ such that $H = \bigcup_{i\in I} \bigcap_{j=1}^{m_i}\{\PP : \PP(A_{ij}) > q_{ij}\}$ for some index set $I$. Because $\Xcal$ and hence $\Pcal(\Xcal)$ are assumed to be separable, the index set $I$ can be taken to be countable.

\subsection{Convergence of empirical measures}\label{sec:uniform_gc}
For an i.i.d.\ sample $X_1, \ldots, X_n \sim \PP$, the \emph{empirical measure} is $\PP_n^X := \frac{1}{n}\sum_{i=1}^n \delta_{X_i}$. By Varadarajan's theorem (\citealp{dudley2002real}, Theorem 11.4.1) we have $\PP_n^X \wto \PP$ almost surely; equivalently $d_{BL}(\PP_n^X, \PP)\to 0$ almost surely, and by dominated convergence the convergence also holds in expectation.

For Theorem \ref{thm:testable_iff_weakly_separated} below we require the convergence to hold uniformly over a class of distributions: a set $W\subseteq \Pcal(\Xcal)$ satisfies the \emph{uniform Glivenko--Cantelli property} if
\begin{equation}\label{eqn:uniform_gc}
	\lim_{n\to\infty}\sup_{\PP \in W}\EE_\PP\!\left[d_{BL}(\PP_n^X, \PP)\right] = 0.
\end{equation}
This property is satisfied whenever $W$ is uniformly tight (and so in particular whenever $W$ is precompact in the weak topology, by Prokhorov's theorem), but also for some non-tight families such as Gaussian location-scale families with bounded variance.

\section{Preliminaries on consistent hypothesis testing and error control}\label{sec:testing_framework}
First, we provide various definitions of consistency and error control. Although these notions themselves are not equivalent, it turns out that various equivalences can be derived for the \emph{existence} of a \mbox{(FP-)}test with these properties.
The definitions and results in this section apply to both FP-tests $\varphi_n : \Xcal^n \to \{0,1,2\}$ (with $\{\varphi_n=0\}$ and $\{\varphi_n=1\}$ open) and binary tests $\varphi_n : \Xcal^n \to \{0,1\}$ without any conditions on the critical regions except measurability.


\begin{definition}\label{def:test_properties}
	Let $(\varphi_n)_{n \in \NN}$ be a (FP-)testing sequence for hypotheses $H_0, H_1 \subseteq \Pcal(\Xcal)$. For $\PP \in H_i$ ($i \in \{0,1\}$) and $n \in \NN$, define the \emph{error processes}:\footnote{The event $\{\varphi_n\neq i\}$ for $\PP\in H_i$ encompasses both the wrong decision $\{\varphi_n = 1 - i\}$ and the suspension region $\{\varphi_n = 2\}$. Strictly speaking, the suspension region is not an error but a refusal to commit to a verdict; we nevertheless refer to the processes $e_n^\star$ as ``error processes'' for ease of exposition.}
	\begin{align*}
		e_n^w(\varphi, \PP)      & := \PP^n(\varphi_n \neq i),                         &  & \text{(weak)}     \\
		e_n^s(\varphi, \PP)      & := \PP^\infty(\exists m \geq n : \varphi_m \neq i), &  & \text{(strong)}   \\
		e_n^\Sigma(\varphi, \PP) & := \sum_{m=n}^\infty \PP^m(\varphi_m \neq i).       &  & \text{(summable)}
	\end{align*}
	For $(\star_0, \star_1) \in \{(\text{weak}, w), (\text{strong}, s), (\text{summable}, \Sigma)\}$ and $i \in \{0,1\}$, we define the following properties:
	\begin{itemize}
		\item $(\varphi_n)$ is \emph{$\star_0$-consistent under $H_i$} if
		\begin{equation*}
			\lim_{n \to \infty} e_n^{\star_1}(\varphi, \PP) = 0;\label{def:test_properties_consistent_eqn}
		\end{equation*}
		\item $(\varphi_n)$ is \emph{uniformly $\star_0$-consistent under $H_i$} if
		\begin{equation*}
			\lim_{n \to \infty} \sup_{\PP \in H_i} e_n^{\star_1}(\varphi, \PP) = 0;\label{def:test_properties_unif_consistent_eqn}
		\end{equation*}
		\item for given $\alpha > 0$, $(\varphi_n)$ has \emph{asymptotic $\star_0$-level $\alpha$ under $H_i$} if
		\begin{equation*}
			\limsup_{n \to \infty} \sup_{\PP \in H_i} e_n^{\star_1}(\varphi, \PP) \leq \alpha;
		\end{equation*}
		\item for given $\alpha > 0$, $(\varphi_n)$ has $\star_0$-\emph{level $\alpha$ under $H_i$} if
		\begin{equation*}\label{def:test_properties:level_eqn}
			\sup_{\PP \in H_i} e_n^{\star_1}(\varphi, \PP) \leq \alpha \quad \text{for all } n \in \NN.
		\end{equation*}
	\end{itemize}
\end{definition}
These notions are also depicted in Table \ref{tab:errors}. The error processes satisfy for all $\PP$ and $n$
\begin{equation}\label{eqn:error_chain_text}
	e_n^w(\varphi, \PP) \leq e_n^s(\varphi, \PP) \leq e_n^\Sigma(\varphi, \PP),
\end{equation}
where the first inequality follows from $\{\varphi_n \neq i\} \subseteq \{\exists m \geq n : \varphi_m \neq i\}$ and the second from the union bound. In particular, $\lim_n e_n^\Sigma(\varphi, \PP) = 0$ (or equivalently $e_1^\Sigma(\varphi, \PP) < \infty$) implies strong consistency by the Borel--Cantelli lemma.
Since $e_n^s(\varphi, \PP)$ and $e_n^\Sigma(\varphi, \PP)$ are non-increasing in $n$, $s$-level $\alpha$ simplifies to $\sup_\PP e_1^s(\varphi, \PP) \leq \alpha$, and $\Sigma$-level $\alpha$ simplifies to $\sup_\PP e_1^\Sigma(\varphi, \PP) \leq \alpha$.

\begin{table}[!t]
	\caption{Properties of a testing sequence $(\varphi_n)$ under $H_i$ ($i \in \{0,1\}$), for given $\alpha > 0$.}\label{tab:errors}
	\centering\footnotesize
	\renewcommand{\arraystretch}{2.0}
	\begin{tabular}{r@{\quad}c@{\quad}c@{\quad}c}
		\toprule
		 & \textbf{$\star$ = weak}                                                       & \textbf{$\star$ = strong} & \textbf{$\star$ = summable} \\
		\midrule
		\textbf{$\star$-consistent}
		 & $\lim_n e_n^w(\varphi,\PP)= 0$
		 & $\lim_n e_n^s(\varphi,\PP)= 0$
		 & $\lim_n e_n^\Sigma(\varphi,\PP)= 0$
		\\
		\midrule
		\textbf{Unif.\ $\star$-consistent}
		 & $\displaystyle\lim_n \sup_{\PP\in H_i} e_n^w(\varphi,\PP)= 0$
		 & $\displaystyle\lim_n \sup_{\PP\in H_i} e_n^s(\varphi,\PP)= 0$
		 & $\displaystyle\lim_n \sup_{\PP\in H_i} e_n^\Sigma(\varphi,\PP)= 0$
		\\
		\textbf{Asymp.\ $\star$-level $\alpha$}
		 & $\displaystyle\limsup_{n}\sup_{\PP\in H_i} e_n^w(\varphi,\PP)\leq\alpha$
		 & $\displaystyle\limsup_{n}\sup_{\PP\in H_i} e_n^s(\varphi,\PP)\leq\alpha$
		 & $\displaystyle\limsup_{n}\sup_{\PP\in H_i} e_n^\Sigma(\varphi,\PP)\leq\alpha$
		\\
		\textbf{$\star$-level $\alpha$}
		 & $\displaystyle\sup_{\PP\in H_i} e_n^w(\varphi,\PP)\leq\alpha\ \forall n$
		 & $\displaystyle\sup_{\PP\in H_i} e_1^s(\varphi,\PP)\leq\alpha$
		 & $\displaystyle\sup_{\PP\in H_i} e_1^\Sigma(\varphi,\PP)\leq\alpha$
		\\
		\bottomrule
	\end{tabular}
\end{table}

A testing sequence $(\varphi_n)_{n\in\NN}$ is called \emph{weakly (strongly) consistent} if it is weakly (strongly) consistent under both $H_0$ and $H_1$. Hypotheses for which a consistent (FP-)test exists are sometimes referred to as \emph{discernible} \citep{dembo1994topological}. The following lemma gives a useful alternative characterisation of strong consistency.

\begin{lemma}\label{thm:strong_consistency_equiv}
	A testing sequence $(\varphi_n)$ is strongly consistent if and only if for all $i\in\{0,1\}$ and all $\PP\in H_i$,
	\begin{equation*}
		\PP^\infty(\varphi_n \neq i \text{ for finitely many } n) = 1,
	\end{equation*}
	or equivalently, $\PP^\infty(\liminf_n\{\varphi_n = i\}) = 1$.
\end{lemma}
\begin{proof}
	The events $A_n := \{\exists m \geq n : \varphi_m \neq i\}$ are decreasing and we have $\bigcap_n A_n = \limsup_n \{\varphi_n \neq i\}$. By continuity of measures from above, $\PP^\infty(\limsup_n \{\varphi_n \neq i\}) = \lim_{n\to\infty} \PP^\infty(A_n) = \lim_{n\to\infty} e_n^s(\varphi, \PP)$.
\end{proof}

The first non-trivial observation is that the existence of weakly and strongly consistent (FP-)tests are equivalent. For binary tests this was shown by \cite{nobel2006hypothesis}; we adapt his proof to FP-tests.

\begin{theorem}\label{thm:consistency_equiv}
	Given a pair of hypotheses $H_0, H_1\subseteq \Pcal(\Xcal)$, there exists a weakly consistent (FP-)test if and only if there exists a strongly consistent (FP-)test.
\end{theorem}
\begin{proof}
	Any strongly consistent test is also weakly consistent. Conversely, let $\varphi_n:\Xcal^n\to \{0,1,2\}$ be weakly consistent, let $k_n := \lfloor\log(n)\rfloor$ and $m_n := \lfloor n/k_n\rfloor$. For $i\in\{0,1\}$ define the random variable $Y_{n,j}^{i} := \I \{\varphi_{k_n}(X_{k_n j+1}, ..., X_{k_n j + k_n}) = i\}$ such that $Y_{n,0}^i, ..., Y_{n,m_{n}-1}^i$ are i.i.d., and let
	\begin{equation*}
		\psi_n(X_1, ..., X_n) := \begin{cases}
			0 & \text{if } \frac{1}{m_n}\sum_{j=0}^{m_n - 1}Y_{n,j}^0 > 1/2; \\
			1 & \text{if } \frac{1}{m_n}\sum_{j=0}^{m_n - 1}Y_{n,j}^1 > 1/2; \\
			2 & \text{otherwise}.
		\end{cases}
	\end{equation*}
	(In the case that $\varphi_n$ is a binary test, as corresponding binary test $\psi_n$ is obtained by replacing one of the strict inequalities with a non-strict inequality.)
	Fix $i\in\{0,1\}$. For $\PP \in H_i$ we have $\mu_n^i := \EE_\PP[Y_{n,j}^i] = \PP^n(\varphi_{k_n}(X_{k_n j+1}, ..., X_{k_n j + k_n}) = i) \to 1$ by weak consistency of $\varphi_n$, and hence for every $\varepsilon \in (0, 1/2)$ there is an $N_\PP$ such that $1/2 \leq \mu_n^i - \varepsilon$ for all $n\geq N_\PP$, so using Hoeffding's inequality this gives
	\begin{multline*}
		\PP^n(\psi_n \neq i)
		= \PP^n\left(\frac{1}{m_n}\sum_{j=0}^{m_n - 1}Y_{n,j}^i \leq 1/2\right)
		\leq \PP^n\left(\frac{1}{m_n}\sum_{j=0}^{m_n - 1}(Y_{n,j}^i - \mu_n^i) \leq -\varepsilon\right)\\
		\leq \PP^n\left(\left|\frac{1}{m_n}\sum_{j=0}^{m_n - 1}(Y_{n,j}^i - \mu_n^i)\right| > \varepsilon\right)
		\leq 2e^{-2m_n \varepsilon^2},
	\end{multline*}
	we get $\sum_{n\geq N_\PP}e^{-2m_n \varepsilon^2} < \infty$, hence $\sum_{n=1}^\infty \PP^n(\psi_n \neq i) < \infty$, so the result follows from the Borel-Cantelli lemma.
\end{proof}

A second equivalence concerns the various notions of uniform error control in Definition \ref{def:test_properties}. Similar as for pointwise consistency, a testing sequence $(\varphi_n)_{n\in\NN}$ is called \emph{uniformly weakly (strongly) consistent} if it is uniformly weakly (strongly) consistent under both $H_0$ and $H_1$. For binary tests, the equivalence of uniform weak consistency and uniform strong consistency was shown by \cite{pfanzagl1968existence}.
\begin{theorem}\label{thm:testable_equiv_error}
	Let a pair of hypotheses $H_0, H_1\subseteq \Pcal(\Xcal)$ be given, and let $K\subseteq \{H_0, H_1\}$ be a set of hypotheses.
	The following are equivalent, for any $\star,\star' \in \{w, s, \Sigma\}$:
	\begin{enumerate}[label=(\alph*)]
		\item\label{thm:error_control:asymptotic_level} for every $\alpha > 0$, there exists a consistent (FP-)testing sequence $(\varphi_n)$ with asymptotic $\star$-level $\alpha$ under $H_i$ for all $H_i \in K$:
		\begin{align*}
			\limsup_{n\to\infty}\sup_{\PP\in H_i} e_n^{\star}(\varphi, \PP) \leq \alpha;
		\end{align*}
		\item\label{thm:error_control:unif_consistent} there exists a consistent (FP-)testing sequence $(\varphi_n)$ which is uniformly $\star'$-consistent under $H_i$ for all $H_i \in K$:
		\begin{align*}
			\lim_{n\to\infty}\sup_{\PP\in H_i} e_n^{\star'}(\varphi, \PP) = 0.
		\end{align*}
	\end{enumerate}
	If for only a single hypothesis the errors are controlled, say $K = \{H_i\}$, the above are also equivalent to the following, for any $\star'' \in \{w, s, \Sigma\}$:
	\begin{enumerate}[label=(\alph*),resume]
		\item\label{thm:error_control:level} for every $\alpha > 0$ there exists a consistent (FP-)testing sequence $(\varphi_n)$ with $\star''$-level $\alpha$ under $H_i$:
		\begin{align*}
			\sup_{\PP \in H_i} e_n^{\star''}(\varphi, \PP) \leq \alpha \quad \text{for all } n \in \NN.
		\end{align*}
	\end{enumerate}
	Throughout, it is equivalent to consider weak or strong consistency of the tests.
\end{theorem}
\begin{proof}
	The chain of inequalities \eqref{eqn:error_chain_text} implies that each condition with $\star = \Sigma$ implies the same condition with $\star = s$, which implies the same condition with $\star = w$. Similarly, strong consistency implies weak consistency. Hence, in each equivalence below, it suffices to show that the weakest version (with $\star = w$ and weak consistency) implies the strongest (with $\star = \Sigma$ and strong consistency).

	\smallskip
	\noindent\emph{Equivalence of \ref{thm:error_control:asymptotic_level} and \ref{thm:error_control:unif_consistent}.}
	The direction $\ref{thm:error_control:unif_consistent} \implies \ref{thm:error_control:asymptotic_level}$ is immediate from \eqref{eqn:error_chain_text}. For the converse, we show that \ref{thm:error_control:asymptotic_level} with $\star = w$ and weak consistency implies \ref{thm:error_control:unif_consistent} with $\star' = \Sigma$ and strong consistency.
	Let $\varepsilon \in (0, 1/2)$ and let $(\varphi_n)$ be weakly consistent with $\sup_{\PP \in H_i}\PP^n(\varphi_n \neq i)\leq 1/2 - \varepsilon$ for all $n\geq N$ and all $H_i \in K$. With $\psi_n$ as defined in the proof of Theorem \ref{thm:consistency_equiv}, $N_\PP = N$ is independent of $\PP$, so $\PP^n(\psi_n \neq i) \leq 2e^{-2m_n \varepsilon^2}$ uniformly over $\PP \in H_i$ for all $n\geq N$. Hence $\sup_{\PP \in H_i}\sum_{n\geq N}\PP^n(\psi_n \neq i)\leq \sum_{n\geq N}2e^{-2m_n \varepsilon^2} < \infty$, and as the tail of a convergent series, $\sup_{\PP \in H_i} e_n^\Sigma(\psi, \PP) \to 0$.

	\smallskip
	\noindent\emph{Equivalence with \ref{thm:error_control:level} when $K = \{H_i\}$.}
	Without loss of generality, let $K = \{H_0\}$. The direction $\ref{thm:error_control:level} \implies \ref{thm:error_control:asymptotic_level}$ is immediate from \eqref{eqn:error_chain_text}. For the converse, we show that \ref{thm:error_control:unif_consistent} with $\star' = \Sigma$ implies \ref{thm:error_control:level} with $\star'' = \Sigma$: for every $\alpha > 0$ there is an $N$ such that $\sup_{\PP \in H_0} e_n^\Sigma(\varphi, \PP) \leq \alpha$ for all $n\geq N$. Then $\psi_n := \I\{n \geq N\}\varphi_n$ satisfies $\sup_{\PP \in H_0} e_n^\Sigma(\psi, \PP) \leq \alpha$ for all $n$.
\end{proof}

By Theorem \ref{thm:testable_equiv_error}, all notions of uniform error control in Definition \ref{def:test_properties} collapse into a single concept.
\begin{definition}\label{def:unif_error_control}
	We say that a pair of hypotheses $H_0$ and $H_1$ are \emph{consistently FP-testable} if there exists a weakly (or equivalently strongly, by Theorem \ref{thm:consistency_equiv}) consistent FP-test for them, and that they are \emph{consistently FP-testable with uniform error control under $H_i$} if there exists a consistent FP-test that satisfies any of the equivalent conditions of Theorem \ref{thm:testable_equiv_error}.
\end{definition}

\section{Main results}\label{sec:characterisations}

Let $\Xcal$ be a separable metric space, and let $H_0, H_1$ be disjoint sets of Borel probability measures on $\Xcal$.
\begin{theorem}\label{thm:discernible_iff_weakly_fsigma}
	The following are equivalent:
	\begin{enumerate}
		\item\label{thm:discernible_iff_weakly_fsigma:test} there exists a consistent FP-test;
		\item\label{thm:discernible_iff_weakly_fsigma:topological_condition} $H_0$ and $H_1$ are $F_\sigma$ in the weak topology on $W :=H_0 \cup H_1$.
	\end{enumerate}
\end{theorem}

\begin{theorem}\label{thm:testable_iff_weakly_closed}
	For $i \in \{0,1\}$, the following are equivalent:
	\begin{enumerate}
		\item \label{thm:testable_iff_weakly_closed:test_weak_error_control}
		there exists a consistent FP-test $\varphi_n$ with
		\begin{align*}
			\lim_{n\to\infty}\sup_{\PP \in H_i}\PP^n(\varphi_n = 1-i) =0;
		\end{align*}
		\item\label{thm:testable_iff_weakly_closed:test_strong_error_control}
		there exists a consistent FP-test $\varphi_n$ with uniform error control under $H_i$, e.g.
		\begin{align*}
			\lim_{n\to\infty}\sup_{\PP \in H_i}\PP^n(\varphi_n \neq i) =0;
		\end{align*}
		\item\label{thm:testable_iff_weakly_closed:topological_condition} $H_i$ is closed in the weak topology on $W := H_0 \cup H_1$;
	\end{enumerate}
\end{theorem}

\begin{theorem}\label{thm:testable_iff_weakly_clopen}
	The following are equivalent:
	\begin{enumerate}
		\item\label{thm:testable_iff_weakly_clopen:test} there exists a consistent FP-test $\varphi_n$ with for all $i\in\{0,1\}$:
		\begin{equation*}
			\lim_{n\to\infty}\sup_{\PP \in H_i}\PP^n(\varphi_n = 1-i) = 0;
		\end{equation*}
		\item\label{thm:testable_iff_weakly_clopen:topological_condition} $H_0$ and $H_1$ are clopen in the weak topology on $W := H_0\cup H_1$.
	\end{enumerate}
\end{theorem}

The next theorem characterises the strongest mode of testability: uniform consistency with bounded measurement precision. Under a precompactness assumption, this is equivalent to metric separation of $H_0$ and $H_1$ in the bounded Lipschitz metric.

\begin{theorem}\label{thm:testable_iff_weakly_separated}
	For the following statements we have the implications \ref{thm:testable_iff_weakly_separated:separation} $\implies$ \ref{thm:testable_iff_weakly_separated:test} $\implies$ \ref{thm:testable_iff_weakly_separated:topological_condition}:
	\begin{enumerate}
		\item\label{thm:testable_iff_weakly_separated:separation} $d_{BL}(H_0, H_1) > 0$ and $W := H_0 \cup H_1$ satisfies the uniform Glivenko--Cantelli property;
		\item\label{thm:testable_iff_weakly_separated:test} there exists a uniformly consistent BP-test, e.g. with for all $i\in\{0,1\}$
		\begin{align*}
			\lim_{n\to\infty}\sup_{\PP \in H_i}\PP^n(\varphi_n \neq i) =0;
		\end{align*}
		\item\label{thm:testable_iff_weakly_separated:topological_condition} $H_0$ and $H_1$ have disjoint closures in $\Pcal(\Xcal)$.
	\end{enumerate}
	If additionally $\Xcal$ is complete and $W$ is precompact in the weak topology, then it satisfies the uniform Glivenko--Cantelli property and \ref{thm:testable_iff_weakly_separated:separation} $\iff$ \ref{thm:testable_iff_weakly_separated:test} $\iff$ \ref{thm:testable_iff_weakly_separated:topological_condition}.
\end{theorem}
The proofs are given below in Section \ref{sec:proofs}.
Throughout one can equivalently consider strong consistency and weak consistency, by Theorem \ref{thm:consistency_equiv}.
Theorem \ref{thm:testable_equiv_error} shows that the uniform error control of Theorem \ref{thm:testable_iff_weakly_closed}, which is now stated as uniform consistency under $H_i$, can equivalently be stated as controlling the level under $H_i$, either asymptotically, or for finite samples. The uniform consistency of Theorem \ref{thm:testable_iff_weakly_separated} can equivalently be stated as control of the asymptotic level under $H_0$ and $H_1$. Similarly, it can be shown that the `uniform consistency' of Theorem \ref{thm:testable_iff_weakly_clopen} can equivalently be stated as asymptotic or finite-sample error control of these specific kinds of errors.


Theorem \ref{thm:testable_iff_weakly_closed} shows that uniform control of $\PP^n(\varphi_n \neq i)$ is equivalent to uniform control of $\PP^n(\varphi_n = 1-i)$. For two-sided error control, the distinction matters: pointwise convergence of the suspension region (Theorem \ref{thm:testable_iff_weakly_clopen}) corresponds to different topological properties of the hypotheses than uniform convergence of the suspension region (Theorem \ref{thm:testable_iff_weakly_separated}).

Theorem \ref{thm:testable_iff_weakly_separated} shows that disjoint closures of the hypotheses is necessary for the existence of a uniformly consistent BP-test. The sufficient conditions are metric separation $d_{BL}(H_0, H_1) > 0$ and the uniform Glivenko--Cantelli property \eqref{eqn:uniform_gc}. Note that disjoint closures do not imply metric separation without precompactness: for instance, $H_0 = \{\delta_k : k\in \NN\}$ and $H_1 = \{\delta_{k+1/k} : k\in\NN\}$ on $\RR$ have disjoint closures but $d_{BL}(H_0, H_1) = 0$. If $H_0\cup H_1$ is precompact (equivalently, uniformly tight by Prokhorov's theorem), disjoint closures implies metric separation and \eqref{eqn:uniform_gc} holds automatically, so the two directions yield an equivalence. The implication \ref{thm:testable_iff_weakly_separated:separation} $\implies$ \ref{thm:testable_iff_weakly_separated:test} also covers non-precompact examples: for instance, one can uniformly consistently test whether the mean of a Gaussian with bounded variance is smaller than $-1$ or larger than $1$, since $d_{BL}(H_0, H_1) > 0$ and the uniform Glivenko--Cantelli property holds.

The topological condition of each theorem implies the topological condition of the previous theorem: metric separation implies the hypotheses to be clopen in the subspace topology, which implies $H_0$ to be closed, which implies $H_0$ and $H_1$ to be $F_\sigma$ (since the weak topology is a separable metric topology).

\smallskip
The four theorems yield FP-tests of differing precision strength. The constructions in the proofs of Theorems \ref{thm:discernible_iff_weakly_fsigma}, \ref{thm:testable_iff_weakly_closed} and \ref{thm:testable_iff_weakly_clopen} produce critical regions that are open and disjoint in $\Xcal^n$, so the verdict can be determined by measurements of (unbounded) finite precision --- the precision required may depend on the location of the sample within the critical region (see Section \ref{sec:finite_precision}). The construction in the proof of Theorem \ref{thm:testable_iff_weakly_separated} is stronger: it is a BP-test with metric separation of the critical regions by $0<\delta < d_{BL}(H_0, H_1)$ that can be taken arbitrarily close to $d_{BL}(H_0, H_1)$.

\smallskip
These results are readily applicable to the example from the introduction, since the parameter space $[0,1]$ with the Euclidean topology is homeomorphic to the set of Bernoulli distributions equipped with the weak topology. The hypothesis $H_1$ in (\ref{eqn:ex:untestable}) is not $F_\sigma$, so there does not exist a consistent FP-test for this problem. The hypotheses from (\ref{eqn:ex:fsigma}) are both $F_\sigma$ so there exists a consistent test, but no form of uniform error control is possible since neither $H_0$ nor $H_1$ is closed. In (\ref{eqn:ex:closed}) $H_0$ is closed, so there exists a consistent FP-test with uniform error control under $H_0$. Both hypotheses of (\ref{eqn:ex:clopen}) are clopen in the relative topology on $W$, so there exists a consistent FP-test with the type of error control as in Theorem \ref{thm:testable_iff_weakly_clopen}, but uniform consistency as in Theorem \ref{thm:testable_iff_weakly_separated} is not feasible. In (\ref{eqn:ex:separated}) the two hypotheses are relatively compact and have disjoint closures in the ambient space, so there exists a uniformly consistent test.

\smallskip
In general, these sufficient conditions for various types of testability are not constructive: the tests of Theorems \ref{thm:discernible_iff_weakly_fsigma}, \ref{thm:testable_iff_weakly_closed} and \ref{thm:testable_iff_weakly_clopen} are constructed by expressing the hypotheses in terms of the subbasis elements $\{\PP\in W : \PP(A)>q\}$ of the weak topology. For example, in Theorem \ref{thm:testable_iff_weakly_closed} one might know that $H_1$ is open without knowing an explicit representation in terms of the subbasis elements.
However, for the examples (\ref{eqn:ex:fsigma}), (\ref{eqn:ex:closed}) and (\ref{eqn:ex:clopen}) these subbasis representations are easily derived, allowing us to give explicit tests, following the constructions in the proofs. We explicitly provide the test for the pair of hypotheses $H_0 = [0,1]\cap \QQ$ and $H_1 = [0,1] \cap (\QQ + \sqrt{2})$ of (\ref{eqn:ex:clopen}) in Example \ref{ex:fsigma} below.
The test of Theorem \ref{thm:testable_iff_weakly_separated} requires the computation of the BL-distance between the empirical measure and the hypotheses, which might be computationally intractable as well.

\smallskip
These sufficient conditions for FP-testability also imply binary testability with corresponding modes of error control, by joining the suspension region $\{\varphi_n=2\}$ with either $\{\varphi_n=0\}$ or $\{\varphi_n=1\}$. In Theorem \ref{thm:discernible_iff_weakly_fsigma}, merging the suspension region with $\{\varphi_n=0\}$ or $\{\varphi_n=1\}$ both give a consistent test. Actually, the topological condition of $H_0$ and $H_1$ being $F_\sigma$ is equivalent to the existence of two binary testing sequences: one with $\{\varphi_n=0\}$ open, and one with $\{\varphi_n=1\}$ open. This can rather straightforwardly be deduced from the proof of Theorem \ref{thm:discernible_iff_weakly_fsigma}.
For the test of clause \ref{thm:testable_iff_weakly_closed:test_strong_error_control} of Theorem \ref{thm:testable_iff_weakly_closed} it does not matter for maintaining error control whether $\{\varphi_n=2\}$ is joined with $\{\varphi_n=0\}$ or $\{\varphi_n=1\}$, but for the test of clause \ref{thm:testable_iff_weakly_closed:test_weak_error_control} the suspension region must be joined with $\{\varphi_n=0\}$ to maintain the error control under $H_0$. It can be shown that $H_0$ being closed is equivalent to the existence of a binary testing sequence with $\{\varphi_n=1\}$ open and uniform error control under $H_0$.
For Theorem \ref{thm:testable_iff_weakly_clopen} the FP-test cannot be converted to a binary test while maintaining the same error control. Finally, for Theorem \ref{thm:testable_iff_weakly_separated} the binary test can be constructed both ways while maintaining uniform consistency.

\smallskip
Results similar to Theorems \ref{thm:discernible_iff_weakly_fsigma} -- \ref{thm:testable_iff_weakly_separated} exist in the literature.
Theorem \ref{thm:discernible_iff_weakly_fsigma} generalises \cite{dembo1994topological} (Theorem 2) who provide the similar characterisation that, under the assumption that all measures in $W$ have uniformly integrable densities, there exists a consistent binary test if and only if the hypotheses are $F_\sigma$ in the weak topology. \cite{genin2017topology} (Theorem 4.3) drop the uniform integrability assumption (but the assumption of having densities remains), for which they show that the $F_\sigma$ condition is equivalent to the existence of a binary testing sequence $(\varphi_n)_{n\in\NN}$ with $\PP(\partial\{\varphi_n = 0\}) = \PP(\partial\{\varphi_n = 1\}) = 0$ for all $\PP\in W$. \cite{ermakov2017consistent} (Theorem 4.4) considers the topology of setwise convergence instead of the weak topology for which he shows -- under the assumption that $W$ is contained in a $\sigma$-compact set -- that the $F_\sigma$ condition is equivalent to the existence of a binary testing sequence, without any regularity of the critical regions.

An analogue of Theorem \ref{thm:testable_iff_weakly_closed} is also shown by \cite{genin2017topology} (Theorem 4.1) under the assumption of having densities, for binary tests with $\PP(\partial\{\varphi_n = 0\}) = \PP(\partial\{\varphi_n = 1\}) = 0$ for all $\PP\in W$. \cite{ermakov2017consistent} (Theorem 4.3) shows that when $H_0$ and $H_1$ are contained in respectively a compact set and a $\sigma$-compact set, consistent binary testability with uniform consistency under $H_0$ is equivalent to $H_0$ being closed and $H_1$ being $F_\sigma$ in the topology of setwise convergence.

Theorem \ref{thm:testable_iff_weakly_clopen} is also given by \cite{genin2018topology} (Theorem 3.2.3) under the assumption of having densities, for binary tests with $\PP(\partial\{\varphi_n = 0\}) = \PP(\partial\{\varphi_n = 1\}) = 0$ for all $\PP\in W$.

Theorem \ref{thm:testable_iff_weakly_separated} is similar to \cite{ermakov2017consistent} (Theorem 4.1), who shows that if $W$ is contained in a compact set in the topology of setwise convergence, then $H_0$ and $H_1$ having disjoint closures is equivalent to the existence of a consistent binary test. Note that compactness in the topology of setwise convergence implies that the weak topology and the topology of setwise convergence coincide, so our results give weaker sufficient conditions for the existence of uniformly consistent tests.

\begin{example}\label{ex:fsigma}
	Denoting $\PP_p = \mathrm{Bernoulli}(p)$, the hypotheses $H_0: p\in [0,1]\cap\QQ$ and $H_1 : p \in [0,1]\cap(\QQ + \sqrt{2})$ can be specified as the countable unions of singletons $H_0 = \{\PP_{p^0_m} : m \in \NN\}$ and $H_1 = \{\PP_{p^1_m} : m \in \NN\}$, where $p^0_m$ and $p^1_m$ are enumerations of $[0,1]\cap\QQ$ and $[0,1]\cap(\QQ + \sqrt{2})$ respectively.
	For testing the closed set $\{\PP_{p^i_m}\}$ against $W\setminus \{\PP_{p^i_m}\}$, Theorem \ref{thm:testable_iff_weakly_closed} gives the test
	\begin{equation*}
		\varphi_{i,n}^m(X_1, ..., X_n) =
		\begin{cases}
			i   & \text{ if }\overline{X}_n \in p_m^i + I_n;    \\
			1-i & \text{ if }\overline{X}_n \notin p_m^i + I_n,
		\end{cases}
	\end{equation*}
	where we write $p+ I_n := (p-t_n^\alpha, p+t_n^\alpha)$ and $t_n^\alpha := \sqrt{\frac{1}{2n}\ln(\pi^2n^2/6\alpha)}$. This gives us the strongly consistent test
	\begin{equation*}
		\varphi_n(X_1, ..., X_n) =
		\begin{cases}
			0 & \text{ if $\exists m \leq n : \overline{X}_n \in p_m^0 + I_n$ and $\overline{X}_n \notin p_k^1 + I_n$ for all $k \leq m$;} \\
			1 & \text{ if $\exists m \leq n : \overline{X}_n \in p_m^1 + I_n$ and $\overline{X}_n \notin p_k^0 + I_n$ for all $k \leq m$;} \\
			2 & \text{ otherwise.}
		\end{cases}
	\end{equation*}
	Another way to see that we indeed have strong consistency, is that if $X_1, X_2, ... \sim \PP_{p^*} \in H_1$, there is an $m$ such that $p_m^1 = p^*$. For sufficiently large $n$ we have $\overline{X}_n \in p^* + I_n$ and $\overline{X}_n \notin p^0_k + I_n$ for all $k \leq m$.
\end{example}

\subsection{Proofs}\label{sec:proofs}
The proofs that the existence of certain tests imply certain topological conditions on the hypotheses are stand-alone. There is however a dependency between the proofs that certain topological conditions imply the existence of FP-tests with the desired properties.
We first prove Theorem \ref{thm:testable_iff_weakly_closed}. From that one, we can build under $F_\sigma$ conditions consistent tests in Theorem \ref{thm:discernible_iff_weakly_fsigma}, and under clopen conditions we can build tests with the two-sided error control as required in Theorem \ref{thm:testable_iff_weakly_clopen}. To prove Theorem \ref{thm:testable_iff_weakly_separated}, we use an entirely different proof strategy using uniform convergence of empirical measures.

\subsubsection{Proof of Theorem \ref{thm:testable_iff_weakly_closed}}
We prove the result for $i=0$, corresponding to uniform error control under $H_0$, which is considered to be weakly closed. The main argument for the existence of a test with uniform error control is that if $H_1$ is open, we can write it as a countable union of finite intersections of subbasis elements $\{\PP \in W : \PP(A) > q\}$. For each disjoint pair of subbasis elements there exists a test with the required error control. The proofs largely follow the structure of the proof of \cite{genin2017topology}, Theorem 4.1.

\begin{lemma}\label{thm:subbasis_is_testable}
	Let $W\subseteq \Pcal(\Xcal)$ be given, let $A\subseteq \Xcal$ be open and let $q\in[0,1]$. For the hypotheses $H_0 := \{\PP\in W : \PP(A) \leq q\}$ and $H_1 = \{\PP \in W : \PP(A) > q\}$, for every $\alpha > 0$ there exists a strongly consistent FP-testing sequence $(\varphi_n)_{n\in\NN}$ with $\sup_{\PP\in H_0}\sum_{n=1}^\infty\PP^n(\varphi_n \neq 0) \leq \alpha$.
\end{lemma}
\begin{proof}
	Let $\alpha > 0$ be given, let $A^c_{1/n} := \bigcup_{a\in A^c}B(a, 1/n)$ be the $1/n$-neighborhood of $A^c$, and define
	\begin{equation*}
		\varphi_n(X_1, ..., X_n) :=
		\begin{cases}
			0 & \text{ if }\frac{1}{n}\sum_{i=1}^n\I_{A^c_{1/n}}(X_i) > 1 - q - t_n;              \\
			1 & \text{ if }\frac{1}{n}\sum_{i=1}^n\I_{\mathrm{ext}(A^c_{1/n})}(X_i) \geq q + t_n; \\
			2 & \text{otherwise},
		\end{cases}
	\end{equation*}
	where $t_n := \sqrt{\frac{1}{2n}\ln(\pi^2 n^2/6\alpha)}$.

	For $\PP\in H_0$ we have $\PP(A) \leq q$, so the fact that $A^c_{1/n}\supseteq A^c$ and Hoeffding's inequality gives
	\begin{multline*}
		\PP^n(\varphi_n \neq 0) = \PP^n\left(\frac{1}{n}\sum_{i=1}^n\I_{A^c_{1/n}}(X_i) \leq 1 - q - t_n\right) \leq 	\PP^n\left(\frac{1}{n}\sum_{i=1}^n\I_{A^c}(X_i) \leq 1 - q - t_n\right)   \\
		= 	\PP^n\left(\frac{1}{n}\sum_{i=1}^n\I_{A}(X_i) \geq q + t_n\right)
		\leq 	\PP^n\left(\frac{1}{n}\sum_{i=1}^n\I_{A}(X_i) \geq \PP(A) + t_n\right)
		\leq e^{-2nt_n^2} = \frac{6\alpha}{\pi^2n^2},
	\end{multline*}
	hence $\sum_{n=1}^\infty\PP^n(\varphi_n \neq 0) \leq \alpha$. Strong consistency under the null then follows from the Borel-Cantelli lemma.

	If $\PP\in H_1$ then $\PP(A) > q$. Since $A= \bigcup_n \mathrm{ext}(A^c_{1/n})$, we have that   $\PP(\mathrm{ext}(A^c_{1/n})) \uparrow \PP(A).$ Therefore, there is an $M\in\NN$ such that $\PP(\mathrm{ext}(A^c_{1/M})) > q$. Since $\mathrm{ext}(A^c_{1/n}) \supseteq \mathrm{ext}(A^c_{1/M})$ for all $n\geq M$ we have for all these $n$ that $\frac{1}{n}\sum_{i=1}^n \I_{\mathrm{ext}(A^c_{1/n})}(X_i) \geq \frac{1}{n}\sum_{i=1}^n \I_{\mathrm{ext}(A^c_{1/M})}(X_i)$, which by the strong law of large numbers converges almost surely to $\PP(\mathrm{ext}(A^c_{1/M})) > q$. Combined with the fact that $t_n\downarrow 0$ we almost surely have that $\frac{1}{n}\sum_{i=1}^n \I_{\mathrm{ext}(A^c_{1/n})}(X_i) \geq q + t_n$ for all sufficiently large $n$, so $\varphi_n \to 1$ a.s.

	To prove that $\{\varphi_n = 0\}$ is open, let for $\gamma \in \{0,1\}^{n}$ the set $(A^c_{1/n})^\gamma \subseteq \Xcal^n$ be the Cartesian product of $n$ sets, where the $i$-th set is $A^c_{1/n}$ if $\gamma_i=1$, and $\Xcal$ if $\gamma_i=0$. For example, $\gamma = (1, ..., 1, 0)$ gives $(A^c_{1/n})^\gamma = A^c_{1/n}\times ... \times A^c_{1/n}\times \Xcal$. Since $A^c_{1/n}$ is open, so is $(A^c_{1/n})^\gamma$, hence $\{\varphi_n = 0\} = \bigcup \{(A^c_{1/n})^\gamma: \gamma \in \{0,1\}^{n}, |\gamma| > n(1 - q - t_n)\}$ is open as well. Since $\mathrm{ext}(A^c_{1/n})$ is open, we have analogously that $\{\varphi_n = 1\}$ is open as well.
\end{proof}

The following lemma shows that for pairs of hypotheses which are consistently testable with uniform error control under the null, the alternative hypotheses enjoy a topological structure: they are closed under countable unions and finite intersections.
For convenience in proving Theorem \ref{thm:testable_iff_weakly_closed}, the result is stated in terms of strong consistency and $\Sigma$-level $\alpha$ (Definition \ref{def:test_properties}), but by Theorem \ref{thm:testable_equiv_error} also holds for weak consistency and $w$-level $\alpha$.
\begin{lemma}\label{thm:testable_is_topology}
	Let $\{(H_0^{ij}, H_1^{ij}) : (i,j)\in\NN^2\}$ be pairs of disjoint hypotheses
	such that
	for every pair $(H_0^{ij}, H_1^{ij})$ and every $\alpha^{ij} > 0$ there exists a strongly consistent FP-test $\varphi_n^{ij}$ with $\sup_{\PP \in H_0^{ij}}\sum_{n=1}^\infty \PP^n(\varphi_n^{ij} \neq 0) \leq \alpha^{ij}$, then for the hypotheses $H_0 := \bigcap_{i=1}^\infty \bigcup_{j=1}^{m_i} H_0^{ij}$ and $H_1 := \bigcup_{i=1}^\infty \bigcap_{j=1}^{m_i} H_1^{ij}$ and any $\alpha > 0$ there exists a strongly consistent FP-test $\varphi_n$ with $\sup_{\PP \in H_0}\sum_{n=1}^\infty\PP^n(\varphi_n \neq 0) \leq \alpha$.
\end{lemma}
\begin{proof}
	Let $\alpha > 0$ be given. For the hypotheses $(H_0^{ij}, H_1^{ij})$, let $\varphi_n^{ij}$ be a strongly consistent FP-test with $\sum_{n=1}^\infty \PP^n(\varphi_n^{ij} \neq 0) \leq \alpha / 2^{i}$. Let $\varphi_n$ be the FP-test defined by
	\begin{align*}
		\{\varphi_n=0\}   & :=\bigcap_{i=1}^n \bigcup_{j=1}^{m_i} \{\varphi_n^{ij}=0\}   \\
		\{\varphi_n=1\}   & :=\bigcup_{i=1}^n \bigcap_{j=1}^{m_i} \{\varphi_n^{ij}=1\}   \\
		\{\varphi_n = 2\} & := \Xcal^n \setminus (\{\varphi_n=0\} \cup \{\varphi_n=1\}),
	\end{align*}
	then $\varphi_n$ has uniform error control under the null: if $\PP \in H_0$, then for every $i\in\NN$ there is a $j_i \leq m_i$ such that $\PP \in H_0^{i{j_i}}$, and since $\{\varphi_n\neq 0\} = \bigcup_{i=1}^n \bigcap_{j=1}^{m_i} \{\varphi_n^{ij}\neq 0\}\subseteq \bigcup_{i=1}^\infty \{\varphi_n^{i{j_i}}\neq 0\}$ we have
	\begin{align*}
		\sum_{n=1}^\infty\PP^n(\varphi_n \neq 0) & \leq \sum_{n=1}^\infty\PP^n\left(\cup_{i=1}^\infty\{\varphi_n^{ij_i} \neq 0\}\right) \leq \sum_{n=1}^\infty\sum_{i=1}^\infty\PP^n(\varphi_n^{ij_i} \neq 0) \leq \alpha.
	\end{align*}
	The Borel-Cantelli lemma then gives strong consistency under $H_0$.

	For every $\PP \in H_1$ there is an $i \in \NN$ such that $\PP\in \bigcap_{j=1}^{m_i}H_1^{ij}$. Because of the inclusion $\cap_{j=1}^{m_i}\{\varphi_n^{ij}=1\}\subseteq \{\varphi_n=1\}$ and the Fréchet inequality\footnote{For measurable sets $A_1, ... , A_n$ we have $\PP(\cap_{i=1}^n A_i) \geq \sum_{i=1}^n\PP(A_i) - n+1$.} we have for $n\geq i$ that
	\begin{multline*}
		\PP^\infty(\liminf_{n}\{\varphi_n = 1\}) = \PP^\infty(\liminf_{n\geq i}\{\varphi_n = 1\}) \\
		\geq \PP^\infty(\liminf_{n}\cap_{j=1}^{m_i}\{\varphi_n^{ij} = 1\}) \geq \sum_{j=1}^{m_i} \PP^\infty(\liminf_{n}\{\varphi_n^{ij} = 1\}) - m_i + 1 = 1,
	\end{multline*}
	so $\varphi_n$ is strongly consistent.
\end{proof}

We are now able to prove Theorem \ref{thm:testable_iff_weakly_closed}:
\begin{proof}
	\textcolor{white}{.}\\
	\vspace{-12pt}
	\item[\ref*{thm:testable_iff_weakly_closed:test_weak_error_control} $\implies$ \ref*{thm:testable_iff_weakly_closed:topological_condition}]
	By Theorem \ref{thm:testable_equiv_error}, \ref*{thm:testable_iff_weakly_closed:test_weak_error_control} is equivalent to the existence (for every $\alpha >0$) of a weakly consistent FP-test with $\sup_{\PP \in H_0}\PP^n(\varphi_n=1)\leq \alpha$ for all $n$.
	If $H_0$ is not closed, then there is a $\QQ \in H_1$ and a sequence $\{\PP_m\}_{m\in\NN}$ in $H_0$ such that $\PP_m \wto \QQ$.  The map $\PP \mapsto \PP^n$ is continuous (\citealp{billingsley1999convergence}, Theorem 2.8) hence $\PP_m^n\wto \QQ^n$. Let $\varphi_n$ be an FP-test with $\sup_{\PP \in H_0}\PP^n(\varphi_n=1)\leq \alpha$ for all $n$, then $\{\varphi_n=1\}$ is open so by the Portmanteau theorem we have $\QQ^n(\varphi_n = 1) \leq \liminf_m \PP_m^n(\varphi_n=1) \leq \alpha$, and hence $\varphi_n$ cannot be weakly consistent at $\QQ$.
	\item[\ref*{thm:testable_iff_weakly_closed:topological_condition} $\implies$ \ref*{thm:testable_iff_weakly_closed:test_strong_error_control}] If $H_1$ is open, then we can write $H_0$ and $H_1$ as
	\begin{align*}
		 & H_0 = \bigcap_{i=1}^\infty \bigcup_{j=1}^{m_i} \{\PP \in W : \PP(A_{ij}) \leq q_{ij}\}
		 &                                                                                        & H_1 = \bigcup_{i=1}^\infty \bigcap_{j=1}^{m_i} \{\PP \in W : \PP(A_{ij}) > q_{ij}\}
	\end{align*}
	for $A_{ij} \subseteq \Xcal$ open and $q_{ij} \in [0,1]$.
	By Lemma \ref{thm:subbasis_is_testable}, for each of these pairs of hypotheses $H_0^{ij} := \{\PP \in W : \PP(A_{ij}) \leq q_{ij}\}$ and $H_1^{ij} := \{\PP \in W : \PP(A_{ij}) > q_{ij}\}$, for each $\alpha^{ij} > 0$ there exists a strongly consistent FP-test $\varphi_n^{ij}$ with $\sup_{\PP\in H_0}\sum_{n=1}^\infty\PP^n(\varphi^{ij}_n \neq 0) \leq \alpha^{ij}$, and by Lemma \ref{thm:testable_is_topology} this implies that for $H_0, H_1$ there exists a strongly consistent FP-test $\varphi_n$ with $\sup_{\PP\in H_0}\sum_{n=1}^\infty\PP^n(\varphi_n \neq 0) < \alpha$. The result now follows from Theorem \ref{thm:testable_equiv_error}.
	\item[\ref*{thm:testable_iff_weakly_closed:test_strong_error_control} $\implies$ \ref*{thm:testable_iff_weakly_closed:test_weak_error_control}] Immediate.
\end{proof}

\subsubsection{Proof of Theorem \ref{thm:discernible_iff_weakly_fsigma}}\label{sec:discernibility}
We now show that consistent FP-testability is equivalent to the hypotheses being $F_\sigma$ in $W$.
\begin{proof}
	\textcolor{white}{.}\\
	\vspace{-12pt}
	\item[\ref*{thm:discernible_iff_weakly_fsigma:test} $\implies$ \ref*{thm:discernible_iff_weakly_fsigma:topological_condition}] Let $\varphi_n$ be a weakly consistent FP-test. For $i \in \{0,1\}$ we have $\PP\in H_i$ if and only if there is an $m\in\NN$ such that $\PP^n(\varphi_n = i) > 2/3$ for all $n\geq m$, hence we can write
	\begin{align*}
		H_0 & = \bigcup_{m=1}^\infty \bigcap_{n=m}^\infty \{\PP\in W : \PP^n(\varphi_n=1) \leq 1/3\}  \\
		H_1 & = \bigcup_{m=1}^\infty \bigcap_{n=m}^\infty \{\PP\in W : \PP^n(\varphi_n=0) \leq 1/3\}.
	\end{align*}
	Since sets of the form $\{\PP \in W : \PP(A) > q\}$ with $A \subseteq \Xcal$ open are open in the weak topology on $W$, and both critical regions $\{\varphi_n=0\}$ and $\{\varphi_n=1\}$ are open, the set $\{\PP\in W : \PP^n(\varphi_n = i) \leq 1/3\}$ is closed, so $H_0$ and $H_1$ are $F_\sigma$ sets in the weak topology.
	\item[\ref*{thm:discernible_iff_weakly_fsigma:topological_condition} $\implies$ \ref*{thm:discernible_iff_weakly_fsigma:test}] If $H_0$ and $H_1$ are disjoint $F_\sigma$ sets, then we can write $H_0 = \bigcup_{m=1}^\infty H_0^m$ and $H_1 = \bigcup_{m=1}^\infty H_1^m$ for closed sets $H_0^m$ and $H_1^m$.

	By Theorem \ref{thm:testable_iff_weakly_closed}, there exists for every $m\in\NN$, for the pair of hypotheses $(H_0^m, W\setminus H_0^m)$ a strongly consistent FP-test $\varphi_{0,n}^m$, where $\varphi_{0,n}^m=0$ corresponds to $H_0^m$ and where $\varphi_{0,n}^m=1$ corresponds to $W\setminus H_0^m$. Similarly, for the pair of hypotheses $(W\setminus H_1^m, H_1^m)$ there exists a strongly consistent FP-test $\varphi_{1,n}^m$ where $\varphi_{1,n}^m=0$ corresponds to $W\setminus H_1^m$ and $\varphi_{1,n}^m=1$ corresponds to $H_1^m$.
	Define the FP-test $\varphi_n$ such that
	\begin{align*}
		\{\varphi_n = 0\} & := \bigcup_{m=1}^n \{\varphi_{0,n}^m = 0\} \bigcap_{k=1}^m \{\varphi_{1,n}^k = 0\} \\
		\{\varphi_n = 1\} & := \bigcup_{m=1}^n \{\varphi_{1,n}^m = 1\} \bigcap_{k=1}^m \{\varphi_{0,n}^k = 1\} \\
		\{\varphi_n = 2\} & := \Xcal^n\setminus (\{\varphi_n = 0\} \cup \{\varphi_n = 1\}).
	\end{align*}
	For $\PP \in H_i$ there is a $m\in\NN$ such that $\PP \in H_i^m$, and note that $\PP \notin H_{1-i}^k$ for all $k\in\NN$, so we obtain
	\begin{multline*}
		\PP^\infty(\liminf_n \{\varphi_n=i\}) = \PP^\infty(\liminf_{n \geq m} \{\varphi_n=i\}) \\
		\geq \PP^\infty(\liminf_n \{\varphi_{i,n}^m = i\} \cap_{k=1}^m \{\varphi_{1-i,n}^k = i\})\\
		\geq \sum_{k=1}^m \PP^\infty(\liminf_n \{\varphi_{i,n}^m = i\} \cap \{\varphi_{1-i,n}^k = i\}) -m+1  \geq 1
	\end{multline*}
	by the Fréchet inequality and the consistency of all $\varphi_{i,n}^m$ and $\varphi_{1-i,n}^k$.
\end{proof}

\subsubsection{Proof of Theorem \ref{thm:testable_iff_weakly_clopen}}
Having Theorem \ref{thm:testable_iff_weakly_closed} at our disposal, we obtain the following characterisation of testability in terms of both hypotheses being clopen in $W$.
\begin{proof}
	\textcolor{white}{.}\\
	\vspace{-12pt}
	\item[\ref*{thm:testable_iff_weakly_clopen:test} $\implies$ \ref*{thm:testable_iff_weakly_clopen:topological_condition}] By Theorem \ref{thm:testable_iff_weakly_closed}, if $H_0$ is not closed in $W$ then there is no consistent test with
	\begin{equation*}
		\lim_{n\to\infty}\sup_{\PP \in H_0}\PP^n(\varphi_n=1) = 0
	\end{equation*}
	and if $H_1$ is not closed in $W$ then there is no test with
	\begin{equation*}
		\lim_{n\to\infty}\sup_{\PP \in H_1}\PP^n(\varphi_n=0) = 0.
	\end{equation*}
	\item[\ref*{thm:testable_iff_weakly_clopen:topological_condition} $\implies$ \ref*{thm:testable_iff_weakly_clopen:test}] By Theorem \ref{thm:testable_iff_weakly_closed} there exists for each $i\in\{0,1\}$ a strongly consistent FP-test $\varphi_n^i$ such that $\lim_{n}\sup_{\PP \in H_i}\PP^n(\varphi_n^i=1-i) = 0$.
	By defining the FP-test
	\begin{align*}
		\{\varphi_n=0\} & := \{\varphi_n^0=0\} \cap \{\varphi_n^1=0\}                 \\
		\{\varphi_n=1\} & := \{\varphi_n^0=1\} \cap \{\varphi_n^1=1\}                 \\
		\{\varphi_n=2\} & := \Xcal^n\setminus (\{\varphi_n=0\} \cup \{\varphi_n=1\}),
	\end{align*}
	we have for both $i\in\{0,1\}$ that $\lim_{n}\sup_{\PP \in H_i}\PP^n(\varphi_n=1-i) \leq \lim_{n}\sup_{\PP \in H_i}\PP^n(\varphi_n^i=1-i) = 0$.
	By Fréchet's inequality we have $\PP^\infty(\liminf_{n}\varphi_n=i) = \PP^\infty(\liminf_{n}\{\varphi_n^0=i\}\cap \liminf_{n}\{\varphi_n^1=i\})\geq 1$ so $\varphi_n$ is strongly consistent.
\end{proof}

\subsubsection{Proof of Theorem \ref{thm:testable_iff_weakly_separated}}
\begin{proof}
	\textcolor{white}{.}\\
	\vspace{-12pt}
	\item[\ref*{thm:testable_iff_weakly_separated:separation} $\implies$ \ref*{thm:testable_iff_weakly_separated:test}] Let $d_{BL}(H_0, H_1) > 0$ and assume \eqref{eqn:uniform_gc}. Let $\gamma < d_{BL}(H_0, H_1)/2$, let $\PP_n^x$ denote the empirical measure at $x \in \Xcal^n$, and let
	\begin{align*}
		\varphi_n(x) := \begin{cases}
			                0 & \text{if $d_{BL}(\PP_n^{x}, H_0) < \gamma$} \\
			                1 & \text{if $d_{BL}(\PP_n^{x}, H_1) < \gamma$} \\
			                2 & \text{otherwise},
		                \end{cases}
	\end{align*}
	then by weak continuity of $x\mapsto \PP_n^x$ the sets $\{\varphi_n=0\}$ and $\{\varphi_n=1\}$ are open and disjoint, and we will verify below that $\varphi_n$ is in fact a BP-test.
	For $\PP \in H_0$ we have $d_{BL}(\PP_n^x, H_0) \leq d_{BL}(\PP_n^x, \PP)$, so $\PP^n(\varphi_n \neq 0) \leq \PP^n(d_{BL}(\PP_n^x, \PP) \geq \gamma)$.
	By \eqref{eqn:uniform_gc}, let $r_n := \sup_{\PP \in H_0\cup H_1}\EE_\PP[d_{BL}(\PP_n^X, \PP)]$, so $r_n\to 0$. For $n$ large enough that $r_n < \gamma/2$:
	\begin{align*}
		\PP^n(\varphi_n\neq 0) & \leq \PP^n(d_{BL}(\PP_n^x, \PP) \geq \gamma)                                                 \\
		                       & \leq \PP^n(d_{BL}(\PP_n^x, \PP) - \EE_\PP\!\left[d_{BL}(\PP_n^X, \PP)\right] \geq \gamma/2).
	\end{align*}
	Since for $x=(x_1, ..., x_n)$ and any $i$ and $x_i'\in\Xcal$ we have with $x' = (x_1, ..., x_i', ..., x_n)$ the bound $\left|d_{BL}(\PP_n^x, \PP) - d_{BL}(\PP_n^{x'}, \PP)\right|\leq 2/n$, McDiarmid's inequality (\citealp{vandervaart2023weak}, Proposition 2.15.3) gives
	\begin{equation*}
		\sup_{\PP \in H_0}\PP^n(\varphi_n\neq 0)\leq \exp(-n\gamma^2/8).
	\end{equation*}
	The argument for $H_1$ is identical, giving uniform consistency.

	We further show that $\{\varphi_n=0\}$ and $\{\varphi_n=1\}$ satisfy the metric-separation condition: there exists $\delta > 0$ such that $d_{\infty}(\{\varphi_n=0\}, \{\varphi_n=1\}) > \delta$ for all $n\in\NN$, with $d_{\infty}$ the $\ell^\infty$ product metric on $\Xcal^n$. For $x\in\{\varphi_n=0\}$ and $y\in\{\varphi_n=1\}$, the triangle inequality gives
	\begin{equation*}
		d_{BL}(\PP_n^x, \PP_n^y) \geq d_{BL}(H_0, H_1) - d_{BL}(\PP_n^x, H_0) - d_{BL}(\PP_n^y, H_1) > d_{BL}(H_0, H_1) - 2\gamma > 0.
	\end{equation*}
	On the other hand, since $|f(x_i) - f(y_i)|\leq d_\Xcal(x_i, y_i)$ for $f\in\BL(\Xcal; \RR)$,
	\begin{equation*}
		d_{BL}(\PP_n^x, \PP_n^y) = \sup_{f\in\BL}\left|\tfrac{1}{n}\sum_{i=1}^n (f(x_i) - f(y_i))\right| \leq \tfrac{1}{n}\sum_{i=1}^n d_\Xcal(x_i, y_i) \leq \max_{1\leq i\leq n} d_\Xcal(x_i, y_i).
	\end{equation*}
	Combining the two bounds, $d_{\infty}(x, y) > d_{BL}(H_0, H_1) - 2\gamma$ for any $x\in\{\varphi_n=0\}$ and $y\in\{\varphi_n=1\}$, so taking $\delta := d_{BL}(H_0, H_1) - 2\gamma > 0$ gives $d_{\infty}(\{\varphi_n=0\}, \{\varphi_n=1\}) > \delta$ for all $n\in\NN$.

	\item[\ref*{thm:testable_iff_weakly_separated:test} $\implies$ \ref*{thm:testable_iff_weakly_separated:topological_condition}] Let $\varphi_n$ be a uniformly consistent BP-test, so $\{\varphi_n=0\}$ and $\{\varphi_n=1\}$ are metrically separated and thus have disjoint closures. For every $\varepsilon>0$ there is an $N>0$ such that $\inf_{\PP \in H_i}\PP^n(\varphi_n = i) \geq 1-\varepsilon$ for all $n\geq N$, for $i=0$ and $i=1$. Suppose for contradiction that $\overline{H_0}\cap\overline{H_1}\neq\emptyset$, and let $\QQ$ be a point in this intersection. Then there are sequences $\PP_m^0$ and $\PP_m^1$ in $H_0$ and $H_1$ respectively converging to $\QQ$. Since the map $\PP\mapsto \PP^n$ is continuous (\citealp{billingsley1999convergence}, Theorem 2.8), we have $(\PP_m^i)^n \wto \QQ^n$. Since $\{\varphi_n=0\}$ and $\{\varphi_n=1\}$ have disjoint closures, the Portmanteau theorem for closed sets gives
	\begin{align*}
		1 & \geq \QQ^n\!\left(\overline{\{\varphi_n=0\}}\right) + \QQ^n\!\left(\overline{\{\varphi_n=1\}}\right)                                 \\
		  & \geq \limsup_m (\PP_m^0)^n\!\left(\overline{\{\varphi_n=0\}}\right) + \limsup_m (\PP_m^1)^n\!\left(\overline{\{\varphi_n=1\}}\right) \\
		  & \geq \limsup_m (\PP_m^0)^n\!\left(\{\varphi_n=0\}\right) + \limsup_m (\PP_m^1)^n\!\left(\{\varphi_n=1\}\right)                       \\
		  & \geq \inf_{\PP \in H_0} \PP^n\!\left(\{\varphi_n=0\}\right) + \inf_{\PP \in H_1} \PP^n\!\left(\{\varphi_n=1\}\right)                 \\
		  & \geq 2(1-\varepsilon),
	\end{align*}
	so we reach a contradiction for $\varepsilon < 1/2$.

	\vspace{1em}
	For the final claim, note that if $H_0\cup H_1$ is precompact, then disjoint closures implies $d_{BL}(H_0, H_1) > 0$ (since the closures are disjoint compact sets).
	It remains to show that the uniform Glivenko--Cantelli property holds. Since $\Xcal$ is a complete separable metric space, by Prokhorov's theorem $W$ is uniformly tight (\citealp{bogachev2007measurevol2}, Theorem 8.6.2), so for every $\varepsilon > 0$ there is a compact $K\subseteq \Xcal$ such that $\PP(\Xcal\setminus K) \leq \varepsilon$ for all $\PP\in W$. This gives
	\begin{align*}
		d_{BL}(\PP_n^x, \PP) & \leq \sup\left\{\left|\int_K f\diff\PP_n^x - \int_K f \diff\PP\right| : f\in \BL \right\} + \PP_n^x(\Xcal\setminus K) + \PP(\Xcal\setminus K),
	\end{align*}
	so $\EE[d_{BL}(\PP_n^x, \PP)] \leq \EE[\sup\left\{\left|\int_K f\diff\PP_n^x - \int_K f \diff\PP\right| : f\in \BL \right\}] + 2\varepsilon$. Since $\BL(\Xcal; \RR)$ is equicontinuous and bounded, by Ascoli's theorem (\citealp{munkres2014topology}, Theorem 47.1) it is precompact in the topology of uniform convergence on compacta. Restricted to $K$, this gives precompactness of $\BL(K;\RR)$ in the sup-norm, hence for every $\varepsilon >0$ there is a $N_\varepsilon$ and functions $f_1, ..., f_{N_\varepsilon} \in \BL(K;\RR)$ such that for every $f\in \BL(K;\RR)$, $\|f-f_i \|\leq \varepsilon$ for some $i$. The triangle inequality then gives
	\begin{equation*}
		\sup\left\{\left|\int_K f\diff\PP_n^x - \int_K f \diff\PP\right| : f\in \BL \right\} \leq \max\left\{\left|\int_K f_i\diff\PP_n^x - \int_K f_i \diff\PP\right| : i=1, ..., N_\varepsilon \right\} + 2\varepsilon.
	\end{equation*}
	By Jensen's inequality applied to the square we obtain
	\begin{align*}
		\EE\left[\left|\int_K f_i\diff\PP_n^x - \int_K f_i \diff\PP\right|\right] & = \EE\left[\left|\frac{1}{n}\sum_{j=1}^n \I_K(X_j)f_i(X_j) - \EE_\PP[\I_K f_i ]\right|\right]           \\
		                                                                          & \leq \sqrt{\mathrm{Var}\left(\frac{1}{n}\sum_{j=1}^n \I_K(X_j)f_i(X_j)\right)} \leq \frac{1}{\sqrt{n}},
	\end{align*}
	so by combining these intermediate results we obtain
	\begin{equation*}
		\EE_\PP\left[d_{BL}(\PP_n^X, \PP)\right] \leq \frac{N_\varepsilon}{\sqrt{n}} + 4\varepsilon,
	\end{equation*}
	so \eqref{eqn:uniform_gc} holds.
\end{proof}

\section{The hardness of conditional independence testing}\label{sec:ci_untestable}
We now turn to applications of the preceding results to conditional independence testing.
\cite{lauritzen2024total} shows that conditional independence is closed under limits in the total variation metric. However, in general, weak convergence does not preserve conditional independence: for a weakly convergent sequence $\PP_n(X, Y, Z)\wto \PP(X, Y, Z)$ with $X\Indep_{\PP_n} Y\given Z$ for all $n\in\NN$, we might have $X\nIndep_{\PP} Y\given Z$.

\begin{example}[\citealp{lauritzen1996graphical}, Example 3.11]\label{ex:gauss}
	Consider a trivariate Gaussian $\PP_n(X, Y, Z)$ with mean zero, and (conditional) covariance matrices given by
	\begin{equation*}
		\Sigma^n = \begin{pmatrix}
			1                  & \frac{1}{2}        & \frac{1}{\sqrt{n}} \\
			\frac{1}{2}        & 1                  & \frac{1}{\sqrt{n}} \\
			\frac{1}{\sqrt{n}} & \frac{1}{\sqrt{n}} & \frac{2}{n}
		\end{pmatrix}
		\quad \text{ and } \quad
		\Sigma_{XY|Z}^n = \begin{pmatrix}
			\frac{1}{2} & 0           \\
			0           & \frac{1}{2}
		\end{pmatrix},
	\end{equation*}
	so we have conditional independence $X\Indep_{\PP_n} Y\given Z$. Since $\PP_n(X, Y, Z)\wto \PP(X, Y, Z)$ with $\PP(X, Y, Z)$ a degenerate multivariate Gaussian distribution with (conditional) covariance matrices
	\begin{equation*}
		\Sigma = \begin{pmatrix}
			1           & \frac{1}{2} & 0 \\
			\frac{1}{2} & 1           & 0 \\
			0           & 0           & 0
		\end{pmatrix}
		\quad \text{ and } \quad
		\Sigma_{XY|Z} = \begin{pmatrix}
			1           & \frac{1}{2} \\
			\frac{1}{2} & 1
		\end{pmatrix},
	\end{equation*}
	we have $X\nIndep_{\PP} Y\given Z$, so conditional independence is not maintained under weak limits. Other examples can be found in \cite{barbie2014topology}
	and \cite{saldi2022geometry}.
\end{example}

Combining this example with Theorem \ref{thm:testable_iff_weakly_closed}, we immediately obtain that for real-valued random variables, conditional independence is not consistently FP-testable with uniform error control under the null. We strengthen this result by allowing for rather general sample spaces, and showing that consistent FP-tests don't exist for conditional independence. We use the following lemmas:

\begin{lemma}\label{thm:ci_dense}
	Let $\Xcal, \Ycal, \Zcal$ be complete separable metric spaces with $\Zcal$ perfect\footnote{That is, $\Zcal$ has no isolated points. Since $\Zcal$ is also Polish, it is uncountable (\citealp{kechris2012classical}, Corollary 6.3).} then $H_0 := \{\PP : X\Indep_\PP Y\given Z\}$ is dense in $\Pcal(\Xcal\times\Ycal\times\Zcal)$.
\end{lemma}
\begin{proof}
	Probability measures of the form $\PP = \sum_{i=1}^n a_i \delta_{k_i}$ with $k_i=(x_i, y_i, z_i) \in \Xcal\times \Ycal\times \Zcal$ are dense in $\Pcal(\Xcal\times \Ycal\times \Zcal)$ (\citealp{bogachev2007measurevol2}, Example 8.1.6). If there are $i \neq j$ such that $z_i=z_j$, then for every $\varepsilon>0$ there is a $z_j'\in\Zcal$ such that $z_j' \neq z_i$ for all $i=1, ..., n$ and $d(z_j, z_j') < \varepsilon / a_j$. Defining $k_j' := (x_j, y_j, z_j')$ and $\PP' := \sum_{i=1}^n a_i \delta_{k_i} + a_j(\delta_{k_j'} - \delta_{k_j})$ we have that $d_{BL}(\PP, \PP') < \varepsilon$. Hence, convex combinations of point-masses with distinct $Z$-coordinates are dense in $\Pcal(\Xcal\times \Ycal\times \Zcal)$. For any such $\PP'$ and any $z_i$ in the support of $\PP'(Z)$ we have $\PP'(X, Y\given Z=z_i) = \delta_{(x_i, y_i)} = \delta_{x_i}\delta_{y_i}$, so $X\Indep_{\PP'} Y\given Z$.
\end{proof}


\begin{lemma}\label{thm:cd_dense}
	Let $\Xcal, \Ycal, \Zcal$ be complete separable metric spaces, then the set $H_1 := \{\PP : X\nIndep_\PP Y\given Z\}$ is dense in $\Pcal(\Xcal\times\Ycal\times\Zcal)$.
\end{lemma}
\begin{proof}
	In \cite{boeken2026are}, Corollary 1 it is shown that $H_1 := \{\PP : X\nIndep_\PP Y\given Z\}$ is dense in $\Pcal(\Xcal\times\Ycal\times\Zcal)$ with respect to the total variation metric. Since $d_{BL}\leq d_{TV}$, the assertion holds.
\end{proof}

Note that in general the two hypotheses can be dense and consistently testable, see e.g.\ Example \ref{ex:fsigma} where we provide a consistent test for the pair of hypotheses $H_0 = [0,1]\cap \QQ$ and $H_1 = [0,1] \cap (\QQ + \sqrt{2})$. However, this only occurs in special cases, since $H_0$ and $H_1$ being dense and $F_\sigma$ implies that $W$ is meager in itself. By the Baire Category Theorem the space $\Pcal(\Xcal\times\Ycal\times\Zcal)$ is not meager in itself, so we obtain the following impossibility result.
\begin{theorem}\label{thm:ci_not_discernible}
	Let $\Xcal, \Ycal, \Zcal$ be complete separable metric spaces with $\Zcal$ perfect, then the hypotheses $H_0 := \{\PP : X\Indep_{\PP} Y\given Z\}$ and $H_1 := \{\PP : X\nIndep_{\PP} Y\given Z\}$ are not consistently FP-testable.
\end{theorem}
\begin{proof}
	Since $\Pcal(\Xcal\times \Ycal\times\Zcal)$ equipped with the weak topology is complete, the Baire Category Theorem implies that it is not meager in itself.
	By Lemmas \ref{thm:ci_dense} and \ref{thm:cd_dense} we have that $H_0$ and $H_1$ are dense.
	If $H_0$ and $H_1$ are both $F_\sigma$, then they are both $G_\delta$ (i.e.\ a countable intersection of open sets). The complement of a dense $G_\delta$ set is meager,
	implying that $H_0$ and $H_1$ are meager. The space $W = H_0 \cup H_1$ would then be meager as well giving a contradiction, so $H_0$ and $H_1$ cannot both be $F_\sigma$. From Theorem \ref{thm:discernible_iff_weakly_fsigma} we conclude that $H_0, H_1$ are not consistently FP-testable.
\end{proof}

The preceding result is for example applicable when $X, Y, Z$ take values in (complete separable metric) function spaces such as $C([0,1], \RR^d)$ \citep{manten2024signature}, $L^p([0,1], \RR^d)$ \citep{lundborg2022conditional} or the Skorohod space $\DD([0,1], \RR^d)$ \citep{boeken2024dynamic}, for example when they represent measurements of continuous-time stochastic processes.

\section{Weak closedness of conditional independence}\label{sec:ci_closed}
Having established that conditional independence is generally not FP-testable, we will now search for conditions under which conditional independence is closed, $F_\sigma$, or metrically separated, implying the existence of a (uniformly) consistent FP-test or BP-test. In particular, we will show that if $\PP_n(X, Y, Z)\wto \PP(X, Y, Z)$ and the conditional distributions $z\mapsto \PP_n(X\given Z=z)$ share a modulus of continuity, then conditional independence $X\Indep Y\given Z$ is maintained in the limit. First, we formalise what is meant by a conditional distribution having a modulus of continuity.

A \emph{Markov kernel} $\PP(X \given Z)$ is a measurable map $\Zcal \to \Pcal(\Xcal)$.\footnote{Here, $\Pcal(\Xcal)$ is equipped with the Borel $\sigma$-algebra generated by the weak topology, which coincides with the smallest $\sigma$-algebra that makes for all $D\in \Bcal(\Xcal)$ the evaluation map $\PP \mapsto \PP(X \in D)$ measurable (\citealp{ghosal2017fundamentals}, Proposition A.5). This definition of a Markov kernel is equivalent to the common definition that for all $D\in \Bcal(\Xcal)$ the map $z\mapsto \PP(X\in D\given Z=z)$ is measurable, and for every $z\in \Zcal$ the map $D\mapsto \PP(X\in D\given Z=z)$ is a probability measure.} For Markov kernels $\PP(X \given Y), \PP(Y\given Z)$, their \emph{product} is defined as the Markov kernel
\begin{equation*}
	\PP(X \given Y) \otimes \PP(Y\given Z) : \Zcal \to \Pcal(\Xcal\times \Ycal), ~ z\mapsto \left(D \mapsto \int_{D}\diff\PP(x \given y)\diff\PP(y \given z)\right)
\end{equation*}
where $D\in \Bcal(\Xcal\times\Ycal)$.
If $\Xcal$ is a complete separable metric space, then for a given joint distribution $\PP(X, Z)$ there exists a Markov kernel (a version of the \emph{conditional distribution}) $\PP(X \given Z)$ such that $\PP(X, Z) = \PP(X \given Z) \otimes \PP(Z)$. This Markov kernel is $\PP(Z)$-almost everywhere uniquely defined.

Throughout, fix a non-decreasing function $\omega : [0,\infty)\to [0,\infty)$ with $\omega(0) = 0$, continuous at $0$, and satisfying $\liminf_{t\downarrow 0}\omega(t)/t > 0$.\footnote{Moduli with $\lim_{t\downarrow 0}\omega(t)/t = 0$ would force any $\omega$-equicontinuous function on a path-connected metric space to be constant, and are excluded.} Our main focus is on distributions $\PP(X, Y, Z)$ for which there exists a version of the conditional distribution $\PP(X\given Z)$ that has \emph{modulus of continuity $\omega$}, i.e.
\begin{equation*}
	d_{BL}(\PP(X \given Z=z), \PP(X \given Z=z')) \leq \omega(d_\Zcal(z, z'))
\end{equation*}
for all $z, z' \in \Zcal$, where $d_\Zcal$ denotes the metric on $\Zcal$ -- we will always take this continuous version of the conditional distribution, which is uniquely determined on the support of $\PP(Z)$. Special cases are $L$-Lipschitz Markov kernels for $\omega(t) = Lt$, Hölder-continuous Markov kernels for $\omega(t) = Lt^\alpha$ and $\alpha\in(0,1]$, log-Hölder Markov kernels for $\omega(t) = L/\log(1/t)^\alpha$, and other regularity classes.
It follows from the definitions that the Markov kernel $z \mapsto \PP(X\given Z=z)$ has modulus of continuity $\omega$ if and only if the conditional expectation $z \mapsto \EE[f(X) \given Z=z]$ has modulus of continuity $\omega$ for all $f\in \BL(\Xcal; \RR)$.

First, we have that the modulus-of-continuity assumption is closed in the weak topology:
\begin{restatable}[]{theorem}{lipmkclosed}
	\label{thm:lip_mk_closed}
	Let $\Xcal, \Zcal$ be complete separable metric spaces with $\Zcal$ locally compact. The set $W_\omega := \{\PP \in \Pcal(\Xcal\times \Zcal) : \PP(X\given Z) \text{ has modulus of continuity } \omega\}$ is closed in the weak topology.
\end{restatable}

Further, a shared modulus of continuity for the conditional $\PP(X\given Z)$ or $\PP(Y\given Z)$ implies that conditional independence is closed in the weak topology.

\begin{restatable}[]{theorem}{independencecontinuous}
	\label{thm:independence_continuous}
	Let $\Xcal, \Ycal, \Zcal$ be complete separable metric spaces with $\Zcal$ locally compact. If $\PP_n(X, Y, Z)\wto \PP(X, Y, Z)$ where $X\Indep_{\PP_n} Y\given Z$ and $\PP_n(X\given Z)$ has modulus of continuity $\omega$ for all $n\in\NN$, then $X\Indep_{\PP} Y\given Z$.
\end{restatable}
The proofs of Theorems \ref{thm:lip_mk_closed} and \ref{thm:independence_continuous} are given in Appendix \ref{sec:lemmas}.
The local compactness hypothesis on $\Zcal$ is satisfied for instance by Euclidean spaces $\RR^d$ and discrete spaces; it excludes infinite-dimensional Banach spaces such as $\ell^p$ or $C([0,1], \RR)$. No local compactness is required of $\Xcal$ (or $\Ycal$ in Theorem \ref{thm:independence_continuous}), so $X$ (c.q.\ $Y$) may take values in such function spaces.

\subsection{Sufficient conditions for Markov kernels with modulus of continuity}
In Section \ref{sec:ci_testing} we will use this modulus-of-continuity assumption for conditional independence testing. To make it better accessible, we investigate sufficient conditions for this assumption to hold.

Scheffé's theorem implies that if a conditional density is continuous in its conditioning variable, so $p(x\given z_n) \to p(x\given z)$ for some sequence $z_n \to z$, then $\PP(X\given Z=z_n) \wto \PP(X\given Z=z)$. The following proposition contains a uniform-continuity analogue of this result.

\begin{proposition}\label{thm:lipschitz_mk_properties}
	Markov kernels with modulus of continuity have the following properties:
	\begin{enumerate}[label=\alph*)]
		\item \label{lem:lipschitz_density_implies_lipschitz_MK} If $\PP(X\given Z)$ has a density $p(x\given z)$ with respect to a finite measure $\QQ(X\in \Xcal) \leq M < \infty$ such that $z\mapsto p(x \given z)$ has modulus of continuity $\omega$ uniformly in $x$, then $\PP(X\given Z)$ has modulus of continuity $M\omega$.

		\item \label{thm:discrete_MK_lipschitz} If $\Zcal$ is discrete, then $\PP(X\given Z)$ is $2$-Lipschitz.


		\item \label{thm:lips_probs_implies_lips_dist} If $z \mapsto \PP(X  \in D\given Z=z)$ has modulus of continuity $\omega$ for all $D \in \Bcal(\Xcal)$, then $z \mapsto \PP(X\given Z=z)$ has modulus of continuity $\omega$ as well.

		\item \label{thm:gaussians_are_lipschitz} If $\PP(X, Z) \sim \Ncal(\mu, \Sigma)$ is multivariate Gaussian, then $\PP(X\given Z)$ has Lipschitz constant $\|\Sigma_{XZ}\Sigma_{ZZ}^{-1}\|_{\mathrm{op}}$.\footnote{Here, $\|\cdot\|_{\mathrm{op}}$ denotes the operator norm.}
	\end{enumerate}
\end{proposition}
\begin{proof}
	\textcolor{white}{.}\\
	\vspace{-12pt}
	\begin{enumerate}[label=\alph*)]
		\item For any $f\in\BL(\Xcal; \RR)$ and $z,z' \in\Zcal$ we have $\left|\EE[f(X) \given z] - \EE[f(X) \given z']\right| \leq \int_{\Xcal}|p(x\given z) - p(x\given z')|\diff\QQ(x) \leq \omega(d_\Zcal(z, z'))\QQ(\Xcal) \leq M\omega(d_\Zcal(z, z'))$.

		\item For any $f\in\BL(\Xcal; \RR)$ and $z,z' \in\Zcal$ with $z\neq z'$ we have $\left|\int f(x)\diff(\PP(x\given z) - \PP(x\given z'))\right| \leq 2 = 2 d_\Zcal(z, z')$, since the discrete metric satisfies $d_\Zcal(z, z') = 1$ whenever $z\neq z'$.


		\item This follows directly from the definition of the total variation metric, which upper-bounds the bounded Lipschitz metric.

		\item Given $\PP(X,Z) \sim \Ncal(\mu, \Sigma)$, the conditional $\PP(X \given Z=z)$ is Gaussian with mean $\mu_{X\given Z}(z) := \mu_X + \Sigma_{XZ} \Sigma_{ZZ}^{-1}(z - \mu_Z)$ and covariance matrix $\Sigma_{X|Z} := \Sigma_{XX} - \Sigma_{XZ}\Sigma_{ZZ}^{-1}\Sigma_{ZX}$. Let $U \sim \Ncal(0, \Sigma_{X| Z})$. For any $f\in\BL(\Xcal; \RR)$ we have
		\begin{align*}
			|\EE[f(X) \given Z=z] - \EE[f(X) \given Z=z']| & = |\EE[f(\mu_{X\given Z}(z) + U)] - \EE[f(\mu_{X\given Z}(z') + U)]| \\
			                                               & \leq \EE[|f(\mu_{X\given Z}(z) + U) - f(\mu_{X\given Z}(z') + U)|]   \\
			                                               & \leq \|\mu_{X\given Z}(z) - \mu_{X\given Z}(z')\|                    \\
			                                               & \leq \|\Sigma_{XZ} \Sigma_{ZZ}^{-1}\|_{\mathrm{op}}\|z- z'\|.
		\end{align*}
	\end{enumerate}
	\vspace{-18.5pt}
\end{proof}
One readily verifies that the conditional distributions $\PP_n(X\given Z)$ and $\PP_n(Y\given Z)$ from Example \ref{ex:gauss} are $\sqrt{n}/2$-Lipschitz and that they satisfy $d_{BL}(\PP_n(X\given z), \PP_n(X\given z + 1/\sqrt{n})) > 1/3$ for every $z\in\RR$ and $n\in\NN$ (by taking the test function $f(x) = \sin(x)$), so the family $\{\PP_n(X\given Z)\}_n$ admits no shared modulus of continuity. Hence, the fact that conditional independence is not maintained in the limit hints at the sharpness of the modulus condition in Theorem \ref{thm:independence_continuous}.

\subsection{Weak closedness of conditional independence via total variation}\label{sec:ci_closed:tv}
We can also take a different approach to finding sufficient conditions under which conditional independence is weakly closed, which does not use uniform continuity of the Markov kernels. Namely, by \cite{lauritzen2024total}, conditional independence is closed in total variation. For sets of probability measures for which the total variation topology and weak topology coincide, we then immediately obtain that conditional independence is weakly closed. To this end, we consider classes of measures with well-behaved densities. Let $\mu$ be a locally finite Borel measure on complete separable metric space $\Xcal$, let $\omega$ be a modulus of continuity, let $M>0$, and let $W_{\mu,\omega,M}$ be the class of distributions with a density $p$ with respect to $\mu$ that are $\omega$-equicontinuous and uniformly bounded by $M$, that is, we have $|p(x)-p(y)| < \omega(|x-y|)$ and $p(x) \leq M$ for all $x, y\in \Xcal$. Extending \cite{boos1985converse} to complete separable metric spaces, \cite{boeken2026are} (Lemma 6) show that the total variation topology and weak topology coincide on $W_{\mu,\omega,M}$, and that $W_{\mu,\omega,M}$ is closed in the weak topology. By \cite{lauritzen2024total}, this implies the following result:

\begin{restatable}[]{theorem}{citvweakclosed}
	\label{thm:ci_tv_weak_closed}
	Let $\Xcal, \Ycal, \Zcal$ be complete separable metric spaces, and let $\mu, M, \omega$ be given. If $\PP_n(X, Y, Z)\wto \PP(X, Y, Z)$ with $X\Indep_{\PP_n} Y\given Z$ and $\PP_n \in W_{\mu,\omega,M}$ for all $n\in\NN$, then $\PP \in W_{\mu,\omega,M}$ as well and $X\Indep_{\PP} Y\given Z$.
\end{restatable}

Another approach would be to consider sufficiently regular exponential families, for which by \cite{barndorff-nielsen2014information}, Section 8.1, Theorem 8.3, the parameter space with the Euclidean topology is homeomorphic to the set of probability measures with the weak topology. By Scheffé's theorem, this topology then coincides with the total variation topology.

The hypotheses of Theorems \ref{thm:independence_continuous} and \ref{thm:ci_tv_weak_closed} are complementary, neither implying the other. Also, Theorem \ref{thm:independence_continuous} cannot be proved by invoking \cite{lauritzen2024total}:
weak convergence plus the modulus-of-continuity assumption does not imply total variation convergence (for example, let $\PP_n(X) = \delta_{1/n}$ and $\PP_n(Y|X=x)=\delta_x$, which is $1$-Lipschitz and $\PP_n(Y|X)\otimes \PP_n(X)$ converges weakly to the point mass on (0,0), but not in total variation).

\subsection{Related literature}\label{sec:ci_closed_related_literature}
As sufficient condition for conditional independence to be closed in the weak topology, \cite{jordan1977continuity} and \cite{hellwig1996sequential} considered probability measures $\PP$ with $z \mapsto \PP(X\given Z=z)$ continuous, but \cite{barbie2014topology} gave a counterexample. They show that conditional independence is closed in the \emph{topology of information} (see also \cite{backhoff2020all}), and show that this topology coincides with the weak topology under similar equicontinuity conditions as Theorem \ref{thm:independence_continuous}.

\section{Sufficient conditions for conditional independence testing}\label{sec:ci_testing}
Having the topological sufficient conditions for testability from Section \ref{sec:characterisations} and the sufficient conditions for conditional independence to be weakly closed from Section \ref{sec:ci_closed} at our disposal, we can relatively easily prove the existence of consistent conditional independence FP-tests, in various settings.

To this end, recall that $W_\omega := \{\PP\in \Pcal(\Xcal\times\Ycal\times\Zcal) : \PP(X\given Z) \text{ has modulus of continuity } \omega\}$. Also, recall the definition of $W_{\mu,\omega,M}$ from Section \ref{sec:ci_closed:tv}.
\begin{theorem}\label{thm:ci_testable}
	Let $\Xcal, \Ycal, \Zcal$ be complete separable metric spaces, let $W\subseteq \Pcal(\Xcal\times\Ycal\times\Zcal)$ and consider the hypotheses $H_0 := \{\PP\in W : X\Indep_{\PP} Y\given Z\}$ and $H_1 := \{\PP\in W : X\nIndep_{\PP} Y\given Z\}$.
	\begin{enumerate}[label=\alph*)]
		\item\label{thm:ci_testable_t1} If $\Zcal$ is locally compact and $W \subseteq W_\omega$ for some modulus of continuity $\omega$, then there exists a strongly consistent FP-test with uniform error control under $H_0$.
		\item\label{thm:ci_testable_t1_discrete} If $\Zcal$ is discrete and $W \subseteq \Pcal(\Xcal\times\Ycal\times\Zcal)$, then there exists a strongly consistent FP-test with uniform error control under $H_0$.
		\item If $W \subseteq W_{\mu,\omega,M}$ for some $(\mu,\omega,M)$, then there exists a strongly consistent FP-test with uniform error control under $H_0$.
		\item\label{thm:ci_testable_t1t2} If $\Zcal$ is locally compact and $W\subseteq W_\omega$ is precompact or satisfies the uniform Glivenko--Cantelli property, and we consider for given $\varepsilon > 0$ the alternative hypothesis
		$$H_1^\varepsilon := \left\{\PP \in W: d_{BL}(\PP(X\given Z) \otimes \PP(Y, Z), \PP(X, Y, Z)) \geq \varepsilon \right\},$$
		then there exists a uniformly consistent BP-test.
		\item\label{thm:ci_testable_consistent} If $\Zcal$ is locally compact and $W\subseteq W_\infty := \bigcup_{n\in\NN} W_{\omega_n}$ for some countable family of moduli $\{\omega_n\}_{n\in\NN}$, then there exists a strongly consistent FP-test.
		\item\label{thm:ci_testable_consistent_tv} If $W\subseteq W_\infty^{TV} := \bigcup_{n\in\NN} W_{\mu_n,\omega_n,n}$ for some countable family $\{(\mu_n, \omega_n)\}_{n\in\NN}$, then there exists a strongly consistent FP-test.
	\end{enumerate}
\end{theorem}
\begin{proof}
	\textcolor{white}{.}\\
	\vspace{-12pt}
	\begin{enumerate}[label=\alph*)]
		\item By Theorem \ref{thm:independence_continuous} we have that $H_0$ is closed in $\Pcal(\Xcal\times\Ycal\times\Zcal)$, hence also in $W_\omega$. The result follows from Theorem \ref{thm:testable_iff_weakly_closed}.
		\item If $\Zcal$ is discrete, then by Proposition \ref{thm:lipschitz_mk_properties}.\ref{thm:discrete_MK_lipschitz} we have that $W_\omega = \Pcal(\Xcal\times\Ycal\times\Zcal)$ for $\omega(t) = 2t$, and $\Zcal$ is locally compact (every singleton is open and compact), so the result follows from part \ref{thm:ci_testable_t1}.
		\item By Theorem \ref{thm:ci_tv_weak_closed}, $H_0$ is closed in $W_{\mu,\omega,M}$, so the result follows from Theorem \ref{thm:testable_iff_weakly_closed}.
		\item From Theorem \ref{thm:independence_continuous} it follows that the map $T : \PP(X, Y, Z) \mapsto \PP(X\given Z)\otimes \PP(Y, Z)$ from $W_\omega$ to $\Pcal(\Xcal\times\Ycal\times\Zcal)$ is continuous. Hence, the map
		\begin{equation*}
			f: W_\omega \mapsto [0,1], \quad\PP(X, Y, Z) \mapsto d_{BL}(\PP(X\given Z) \otimes \PP(Y, Z), \PP(X, Y, Z))
		\end{equation*}
		is continuous as well, so $H_0 \subseteq f^{-1}(\{0\})$ and $H_1^\varepsilon \subseteq f^{-1}([\varepsilon, 1])$ are contained in disjoint closed subsets of $W_\omega$. Since $W_\omega$ is closed in $\Pcal(\Xcal\times\Ycal\times\Zcal)$ by Theorem \ref{thm:lip_mk_closed}, $H_0$ and $H_1^\varepsilon$ are closed in $\Pcal(\Xcal\times\Ycal\times\Zcal)$ as well. If $W$ is precompact or satisfies the uniform Glivenko--Cantelli property, the result follows from Theorem \ref{thm:testable_iff_weakly_separated}.
		\item For each $n\in\NN$, the set $W_{\omega_n}$ is closed in $\Pcal(\Xcal\times\Ycal\times\Zcal)$ by Theorem \ref{thm:lip_mk_closed}, hence also in $W_\infty$. Writing $H_0' = \{\PP \in \Pcal(\Xcal\times\Ycal\times\Zcal) : X\Indep_\PP Y\given Z\}$, we have that $H_0' \cap W_{\omega_n}$ is closed in $W_{\omega_n}$ by Theorem \ref{thm:independence_continuous}, hence also in $W_\infty$. This gives that $\bigcup_{n\in\NN} H_0' \cap W_{\omega_n}$ is $F_\sigma$ in $W_\infty$ and hence $H_0 \subseteq \bigcup_{n\in\NN} H_0' \cap W_{\omega_n}$ is $F_\sigma$ in $W$. Similarly, writing $H_1' = \{\PP \in \Pcal(\Xcal\times\Ycal\times\Zcal) : X\nIndep_\PP Y\given Z\}$ we have that $H_1' \cap W_{\omega_n}$ is open and hence $F_\sigma$ in $W_{\omega_n}$, so by a similar reasoning $H_1$ is $F_\sigma$ in $W$, and the result follows from Theorem \ref{thm:discernible_iff_weakly_fsigma}.
		\item By Theorem \ref{thm:ci_tv_weak_closed}, each set $W_{\mu_n, \omega_n, n}$ is closed in $\Pcal(\Xcal\times\Ycal\times\Zcal)$ and $H_0' \cap W_{\mu_n, \omega_n, n}$ is closed in $W_{\mu_n, \omega_n, n}$, so the result follows from a similar proof as part \ref{thm:ci_testable_consistent}.
	\end{enumerate}
	\vspace{-18.5pt}
\end{proof}

Note that by symmetry, one can also let $\PP(Y\given Z)$ have modulus of continuity $\omega$ in the definition of $W_\omega$ as considered in Theorem \ref{thm:ci_testable}.

\begin{remark}
	Note that the hypotheses $H_0 := W_\omega$ and $H_1 := \Pcal(\Xcal\times\Ycal\times\Zcal)\setminus W_\omega$ are also FP-testable, since $W_\omega$ is closed by Theorem \ref{thm:lip_mk_closed}. Hence, these regularity assumptions for conditional independence testing \emph{are themselves testable}. An assumption-free consistent FP-test with uniform error control under $H_0$ therefore exists for the hypotheses $H_0 := \{\PP\in W_\omega : X\Indep_{\PP} Y\given Z\}$ and $H_1 := \Pcal(\Xcal\times\Ycal\times\Zcal)\setminus H_0$. The same reasoning applies to $W_{\mu,\omega,M}$.
\end{remark}

\subsection{Related literature}\label{sec:ci_related_literature}
\cite{warren2021wasserstein} proposes a conditional independence test based on binning the space $\Zcal$, and proves that it is pointwise asymptotically valid under the null and consistent under the alternative, under the assumption that the distributions $\PP(X, Y\given Z), \PP(X\given Z)$ and $\PP(Y\given Z)$ are $L$-Lipschitz maps from a compact space $\Zcal$ to the space of probability measures equipped with the $p$-Wasserstein distance $W_p$. To compare this assumption to the equicontinuity assumption that we make, we have $d_{BL}(\PP_0, \PP_1) \leq W_p(\PP_0, \PP_1)$, with equality $d_{BL}(\PP_0, \PP_1) = W_1(\PP_0, \PP_1)$ if $\PP_0$ and $\PP_1$ have bounded support (\citealp{bogachev2007measurevol2}, Theorem 8.10.45). Theorem \ref{thm:ci_testable}.\ref{thm:ci_testable_t1} shows that, under the assumptions of \cite{warren2021wasserstein}, there even exists a test that is valid at every sample size $n\in\NN$, and one requires only $\PP(X\given Z)$ or $\PP(Y\given Z)$ to admit a modulus of continuity $\omega$.

\cite{neykov2021minimax} propose a conditional independence test which obtains a minimax-optimal rate, under the assumption that $X, Y, Z$ are compactly supported, have densities, and $\PP(X\given Z)$ and $\PP(Y\given Z)$ are $L$-Lipschitz maps from $\Zcal$ to the space of probability measures equipped with the total variation metric.

\cite{gyorfi2012strongly} propose a conditional independence test and aim at proving its strong consistency without any regularity conditions on $H_0$ and $H_1$, but \cite{neykov2021minimax} point out a mistake in their proof.
\cite{dai2025consistent} propose a consistent test, under the assumption of a smooth parametric model.
For data beyond real-valued random variables, there is some work about conditional independence testing where $X, Y, Z$ take values in function spaces, representing measurements of continuous-time stochastic processes. \cite{lundborg2022conditional} propose an asymptotically valid test under the assumption that $\EE[X \given Z]$ and $\EE[Y\given Z]$ can be estimated sufficiently well, and \cite{manten2024signature} propose a weakly consistent test under the assumption that $\EE[X\given Z], \EE[Y\given Z]$ and $\EE[(X, Y)\given Z]$ can be estimated with certain kernel methods.
None of these tests take any regularity of the critical region into account.

\section{Discussion}
This work establishes a topological framework for understanding the testability of statistical hypotheses under the constraint of finite-precision measurements. Our results reveal that the feasibility of constructing consistent tests with various types of error control can be characterized entirely in terms of the topological properties ($F_\sigma$, closed, clopen, or metric separation) of the null and alternative hypotheses in the weak topology on the space of probability measures.
An important implication is the non-testability of conditional independence hypotheses in general: because both the null and alternative hypotheses are dense and their union is a Baire space, no consistent FP-test exists without additional assumptions. This generalizes and strengthens prior results in the literature, and underscores that conditional independence --- while foundational in causal inference and graphical models --- is not a testable property without regularity conditions.
By imposing a shared modulus-of-continuity assumption on the conditional distributions, we recover testability: we show that this assumption ensures closedness in the weak topology, enabling the construction of tests with uniform error control under $H_0$. This yields sufficient conditions --- including Lipschitz, Hölder, and log-Hölder regularity classes as special cases --- under which conditional independence becomes statistically testable. In particular, we prove that the regularity assumption itself is testable.

Despite these contributions, several important questions remain open.
While the established topological conditions are necessary and sufficient, they might not directly provide one with an explicit test: given an open $H_1$, one must find open $A_{ij} \subseteq \Xcal$ and $q_{ij} \in [0,1]$ such that $H_1 = \bigcup_{i=1}^\infty \bigcap_{j=1}^{m_i}\{\PP : \PP(A_{ij}) > q_{ij}\}$ in order to construct a test, and this representation may be difficult to identify in practice. It would be interesting to find such an explicit representation for conditional independence.
Beyond the existence questions addressed here, the rates of error decay attainable by FP- and BP-tests, and whether our constructions are (minimax) optimal in any sense, remain open.
In the setting of Theorem \ref{thm:testable_iff_weakly_separated}, it would be interesting to study in more practical applications the tradeoff between the level of bounded precision, the separation of the hypotheses, and the rate of uniform convergence.
Another interesting avenue for future research would be to allow for dependent data sampling. By Theorem \ref{thm:testable_equiv_error}, uniform error control under $H_0$ can be interpreted as $\sup_{\PP\in H_0} \PP^\infty(\exists n : \varphi_n\neq 0) \leq \alpha$, which is closely related to anytime-valid p-values, e-values and e-processes \citep{wang2022false,grunwald2024safe}. Since exchangeable sequences of data can be obtained as convex combinations of i.i.d.\ processes \citep{ramdas2022testing}, we conjecture that our results can be extended to anytime-valid tests for exchangeable data. Extensions to topological characterisations of the existence of e-values and e-processes for arbitrary sampling schemes would be of significant interest. It is plausible that analogous topological criteria govern the possibility of valid online or interactive hypothesis testing in complex models.
It remains an open question whether there exists a consistent FP-test for conditional independence if one merely assumes that the distribution has a density.
Finally, from Example \ref{ex:gauss} it seems that some kind of equicontinuity of $\PP(X\given Z)$ or $\PP(Y\given Z)$ is necessary for conditional independence testing with uniform error control under $H_0$. Finding these necessary conditions would be of great interest.

In summary, our results offer a unified topological perspective on the fundamental limits of statistical inference under measurement constraints. They resolve open problems around the testability of nonparametric hypotheses, while opening avenues for further exploration in both theoretical and applied statistics.

\appendix
\section{Proofs of Theorems \ref{thm:lip_mk_closed} and \ref{thm:independence_continuous}}\label{sec:lemmas}

In this appendix, we prove Theorems \ref{thm:lip_mk_closed} and \ref{thm:independence_continuous} from Section \ref{sec:ci_closed}. 

\begin{lemma}\label{thm:disintegration_continuous}
	Let $\Xcal, \Zcal$ be complete separable metric spaces with $\Zcal$ locally compact. If $\PP_n(X, Z) \wto \PP(X, Z)$ and $\PP_n(X\given Z)$ has modulus of continuity $\omega$ for all $n\in\NN$, then $\PP(X\given Z)$ admits a version with modulus of continuity $\omega$ on $\Zcal$, and $\PP_n(X\given Z)\to\PP(X\given Z)$ uniformly on compact subsets of $\Zcal$.
\end{lemma}
\begin{proof}
	The shared modulus of continuity implies that $\{z\mapsto \EE_n[f(X)\given Z=z]\}_{n\in\NN}$ is equicontinuous on $\Zcal$ for every $f\in\BL_1(\Xcal;\RR)$. By \cite{sweeting1989conditional}, Theorem 4, applied with $V=\Zcal$, this equicontinuity together with the joint convergence $\PP_n(X, Z)\wto \PP(X, Z)$ gives $\PP_n(X\given Z=z)\wto \PP(X\given Z=z)$ for every $z\in\Zcal$, where $\PP(X\given Z=z)$ is the unique continuous version of the conditional distribution.\footnote{Although \cite{sweeting1989conditional}, Theorem 4 is stated under the hypothesis that $\Xcal$ is locally compact, the direction (i) $\Rightarrow$ (ii) used here is established via \cite{sweeting1989conditional}, Theorem 3 on the locally compact space $\Zcal$ and does not require local compactness of $\Xcal$.}

	For each $z\in\Zcal$, pointwise weak convergence and Prokhorov's theorem (\citealp{bogachev2007measurevol2}, Theorem 8.6.2) imply that $\{\PP_n(X\given Z=z) : n\in\NN\}$ is relatively compact in $(\Pcal(\Xcal), d_{BL})$. Combined with the equicontinuity, the family $\{z\mapsto \PP_n(X\given z)\}_{n\in\NN}\subseteq C(\Zcal, \Pcal(\Xcal))$ is equicontinuous and pointwise relatively compact, so by Ascoli's theorem (\citealp{munkres2014topology}, Theorem 47.1) it has compact closure in $C(\Zcal, \Pcal(\Xcal))$ equipped with the topology of uniform convergence on compacta. Since $\Zcal$ is locally compact and second countable, the topology of uniform convergence on compact subsets of $\Zcal$ is metrizable (\citealp{munkres2014topology}, Exercise 10 in \S46). Relative compactness in a metrizable space is equivalent to sequential relative compactness, so every subsequence $(n_k)$ has a further subsequence $(n_{k_j})$ converging uniformly on compacta of $\Zcal$ to some continuous limit $g : \Zcal\to\Pcal(\Xcal)$. By the pointwise convergence $\PP_n(X\given Z=z)\wto \PP(X\given Z=z)$ at every $z\in\Zcal$, this limit must satisfy $g(z) = \PP(X\given Z=z)$ for all $z\in\Zcal$. Since every subsequence has a further subsequence converging to the same limit $\PP(X\given Z=\cdot)$, the whole sequence converges to it: $\sup_{z\in K} d_{BL}(\PP_n(X\given z), \PP(X\given z))\to 0$ for every compact $K\subseteq\Zcal$. The limit map $z\mapsto \PP(X\given Z=z)$ inherits modulus of continuity $\omega$ on $\Zcal$ as a uniform-on-compacta limit of $\omega$-equicontinuous maps.
\end{proof}

\lipmkclosed*
\begin{proof}
	Let $\PP_n(X, Z) \wto \PP(X, Z)$ with $\PP_n(X\given Z)$ having modulus of continuity $\omega$ for all $n\in\NN$. By Lemma \ref{thm:disintegration_continuous}, $\PP(X\given Z)$ admits a version with modulus of continuity $\omega$ on $\Zcal$, hence $\PP\in W_\omega$, i.e.\ $W_\omega$ is closed in the weak topology.
\end{proof}

\begin{lemma}\label{thm:disintegration_bound_1}
	Let $\Xcal, \Ycal, \Zcal$ be separable metric spaces. Let $\{\PP_n(X\given Z)\}_{n\in\NN}$ be a family of Markov kernels with shared modulus of continuity $\omega$, and let $\PP_n(Y, Z), \PP(Y, Z)\in\Pcal(\Ycal\times\Zcal)$ with $\PP_n(Y, Z)\wto \PP(Y, Z)$. Then
	\begin{equation*}
		d_{BL}(\PP_n(X\given Z)\otimes \PP_n(Y,Z), \PP_n(X\given Z)\otimes \PP(Y,Z)) \to 0.
	\end{equation*}
\end{lemma}
\begin{proof}
	Considering the product metric $d_{\Xcal\times\Ycal\times\Zcal} = d_\Xcal + d_\Ycal + d_\Zcal$, for any $f\in\BL(\Xcal\times\Ycal\times\Zcal; \RR)$ and $n\in\NN$ define $g_n^f : \Ycal\times\Zcal\to\RR$ by
	\begin{equation*}
		g_n^f(y, z) := \int_\Xcal f(x, y, z) \diff\PP_n(x\given z).
	\end{equation*}
	Then $\|g_n^f\|_\infty \leq 1$, and
	\begin{align*}
		|g_n^f(y, z) - g_n^f(y', z')| & \leq \left|\int (f(x, y, z) - f(x, y', z'))\diff\PP_n(x\given z)\right|                                          \\
		                              & \quad + \left|\int f(x, y', z')\diff(\PP_n(x\given z) - \PP_n(x\given z'))\right|                                \\
		                              & \leq d_\Ycal(y, y') + d_\Zcal(z, z') + d_{BL}(\PP_n(X\given z), \PP_n(X\given z'))                               \\
		                              & \leq d_\Ycal(y, y') + d_\Zcal(z, z') + \omega(d_\Zcal(z, z'))
	\end{align*}
	by the bounded Lipschitz property of $f$ and the shared modulus of continuity of $\{\PP_n(X\given Z)\}$. Hence the family $\mathcal{G} := \{g_n^f : n\in\NN, f\in\BL(\Xcal\times\Ycal\times\Zcal; \RR)\}$ is uniformly bounded by 1 and equicontinuous on $\Ycal\times\Zcal$ with shared modulus of continuity.

	By Fubini's theorem, we have
	\begin{align*}
		&d_{BL}(\PP_n(X\given Z)\otimes \PP_n(Y,Z), \PP_n(X\given Z)\otimes \PP(Y,Z)) \\
		&= \sup_f\int f \diff(\PP_n(X\given Z)\otimes \PP_n(Y, Z) - \PP_n(X\given Z)\otimes \PP(Y, Z)) \\
		&= \sup_f\int g_n^f \diff(\PP_n(Y, Z) - \PP(Y, Z)) \\
		&\leq \sup_{g\in \mathcal{G}}\int g \diff(\PP_n(Y, Z) - \PP(Y, Z)).
	\end{align*}
	Since $\PP_n(Y, Z)\wto \PP(Y, Z)$, this last term converges to 0 by \cite{dudley2002real}, Corollary 11.3.4.
\end{proof}

\begin{lemma}\label{thm:disintegration_bound_2}
	Let $\Xcal, \Ycal, \Zcal$ be separable metric spaces, with $\Zcal$ complete. If $\PP_n(X\given Z) \to \PP(X\given Z)$ uniformly on compacta in $d_{BL}$, then for any given $\PP(Y, Z)$ we have
	\begin{equation*}
		d_{BL}(\PP_n(X\given Z)\otimes \PP(Y,Z), \PP(X\given Z)\otimes \PP(Y,Z)) \to 0.
	\end{equation*}
\end{lemma}
\begin{proof}
	For any $f(x,y,z) \in \BL(\Xcal\times \Ycal \times \Zcal; \RR)$ we have that $x\mapsto f(x,y,z) \in \BL(\Xcal; \RR)$, so
	\begin{align*}
		\Bigg|\int f(x,y,z) \diff \PP_n(x\given z) & - \int f(x,y,z) \diff \PP(x\given z)\Bigg| \leq d_{BL}(\PP_n(X\given z), \PP(X\given z)).
	\end{align*}
	Since $\Zcal$ is separable and complete, the measure $\PP(Z)$ is tight, so for every $\varepsilon>0$ there exists a compact $K_\varepsilon \subseteq \Zcal$ such that $\PP(Z\notin K_\varepsilon) \leq \varepsilon$. Since $d_{BL} \leq 2$, we have
	\begin{align*}
		\Bigg|\int f(x,y,z) \diff \PP_n(x\given z) \diff\PP(y,z) - & \int f(x,y,z) \diff \PP(x\given z)\diff\PP(y,z)\Bigg|                                   \\
		                                                           & \leq \int d_{BL}(\PP_n(X\given z), \PP(X\given z))\diff\PP(z)                           \\
		                                                           & \leq \sup_{z\in K_\varepsilon} d_{BL}(\PP_n(X\given z), \PP(X\given z)) + 2\varepsilon.
	\end{align*}
	Since $\PP_n(X\given Z) \to \PP(X\given Z)$ uniformly on compacta in $d_{BL}$ we get the result.
\end{proof}

\independencecontinuous*
\begin{proof}
	By Lemma \ref{thm:disintegration_continuous}, $\PP_n(X\given Z) \to\PP(X\given Z)$ uniformly on compact subsets of $\Zcal$ in $d_{BL}$. By the chain inequality
	\begin{align*}
		d_{BL}(\PP_n(X\given Z)\otimes \PP_n(Y,Z), & ~\PP(X\given Z)\otimes \PP(Y,Z))                                                       \\
		                                           & \leq d_{BL}(\PP_n(X\given Z)\otimes \PP_n(Y,Z), \PP_n(X\given Z)\otimes \PP(Y,Z)) \\
		                                           & + d_{BL}(\PP_n(X\given Z)\otimes \PP(Y,Z), \PP(X\given Z)\otimes \PP(Y,Z)),
	\end{align*}
	the first term converges to zero by Lemma \ref{thm:disintegration_bound_1} (since $\PP_n(Y, Z)\wto \PP(Y, Z)$, which follows from $\PP_n(X, Y, Z)\wto \PP(X, Y, Z)$), and the second term converges to zero by Lemma \ref{thm:disintegration_bound_2}. Hence $\PP_n(X\given Z) \otimes \PP_n(Y, Z) \wto \PP(X\given Z)\otimes \PP(Y, Z)$. Combined with $\PP_n(X, Y, Z) = \PP_n(X\given Z)\otimes \PP_n(Y, Z)$ and $\PP_n(X, Y, Z) \wto \PP(X, Y, Z)$, this gives $\PP(X, Y, Z) = \PP(X\given Z)\otimes \PP(Y, Z)$, i.e.\ $X\Indep_\PP Y\given Z$.
\end{proof}

\section*{Acknowledgements}
Philip Boeken was supported by Booking.com. Eduardo Skapinakis was funded by national funds through the FCT – Fundação para a Ciência e a Tecnologia, I.P., under the scope of the projects UIDB/00297/2020 (\url{https://doi.org/10.54499/UIDB/00297/2020}) and UIDP/00297/2020 (\url{https://doi.org/10.54499/UIDP/00297/2020}) (Center for Mathematics and Applications) and the FCT scholarship, reference 2022.10596.BD (\url{https://doi.org/10.54499/2022.10596.BD}).
The authors thank Claude Code, Bas Kleijn, Sourbh Bhadane and Johannes Ruf for helpful remarks and suggestions.

\bibliographystyle{apalike}
\bibliography{refs}


\end{document}